\documentclass[11pt]{article}
\usepackage{amsmath, amssymb, amsthm}
\usepackage{enumitem}                                 
\usepackage[margin=1.0in]{geometry}

\AtEndDocument{\bigskip{\footnotesize%
  \textsc{Einstein Institute of Mathematics, Edmond J. Safra Campus, The Hebrew University of Jerusalem,
Givat Ram. Jerusalem, 9190401, Israel.} \par  
  \textit{E-mail address} \texttt{ amir.algom@mail.huji.ac.il
} 
}}

\newtheorem{theorem}{Theorem}[section]

\newtheorem{Counter-example}[theorem]{Counter example}
\newtheorem{Claim}[theorem]{Claim}

\newtheorem{Lemma}[theorem]{Lemma}
\newtheorem{Proposition}[theorem]{Proposition}
\newtheorem{Notation}[theorem]{Notation}
\newtheorem{Definition}[theorem]{Definition}
\newtheorem{Corollary}[theorem]{Corollary}

\newtheorem{Conjecture}[theorem]{Conjecture}

\newtheorem*{theorem*}{Theorem}

\DeclareMathOperator*{\cmpct}{cpct}

\newcommand{\Fomega}{\pi_{m_1} (\tilde{F}_\omega)}
\newcommand{\Eeta}{\pi_{m_2} (\tilde{E}_\eta)}
\newcommand{\Rtheta}{R_{\theta}}
\newcommand{\supp}{\text{supp}}

\newcommand\blfootnote[1]{%
  \begingroup
  \renewcommand\thefootnote{}\footnote{#1}%
  \addtocounter{footnote}{-1}%
  \endgroup
}

\title{Slicing Theorems and rigidity phenomena for self affine carpets}

\author{Amir Algom}
\date{}

\begin{document}
\maketitle

\begin{abstract}
Let\blfootnote{Supported by ERC grant 306494.} $F$ be a Bedford-McMullen carpet defined by independent exponents. We prove that $\overline{\dim}_B (\ell \cap F) \leq \max \lbrace \dim^* F -1,0 \rbrace$ for all lines $\ell$ not parallel to the principal axes, where  $\dim^*$ is Furstenberg's star dimension (maximal dimension of a microset). We also prove several rigidity results for incommensurable Bedford-McMullen carpets, that is, carpets $F$ and $E$ such that all defining exponents are independent: Assuming various conditions, we find bounds on the dimension of the intersection of such carpets, show that self affine measures on them are mutually singular, and prove that  they do not embed affinely into each other.  

We obtain these results as an application of a slicing Theorem for products of certain Cantor sets. This Theorem is a generlization of the results of Shmerkin \cite{shmerkin2016furstenberg}, and Wu \cite{wu2016proof}, that proved Furstenberg's slicing Conjecture \cite{furstenberg1970intersections}.
\end{abstract}

\section{Introduction}
Let $F\subset \mathbb{R}^2$ be a set, and let $\ell \subset \mathbb{R}^2$ be an affine line.  One of the classic questions in geometric measure theory involves studying the dimension of $F\cap \ell$, as we go over all the lines in the plane. It is natural to parametrize a line in the plane by its slope (an element in $\mathbb{R}\cup \lbrace \infty \rbrace$, where $\infty$ corresponds to  lines parallel to the $y$-axis) and its intercept (an element in $\mathbb{R})$.  The most general result in this direction, known as Marstrand's slicing Theorem, asserts that for any fixed slope $u$,
\begin{equation} \label{Eq. Marstrand}
\dim_H F\cap \ell_{u,t} \leq \max \lbrace \dim_H F -1, 0 \rbrace \text{ for Lebesgue almost every } t, 
\end{equation}
where $\dim_H$ denotes the Hausdorff dimension, $\ell_{u,t}$ is the line with slope $u$ and intercept $t$.

While \eqref{Eq. Marstrand} predicts the dimension of the intersection of $F$ with a typical line $\ell$, it is  a  challenging problem to understand the intersection of $F$ with a fixed line $\ell$.  However, when the set $F$ has some arithmetic or dynamical origin, it is sometimes possible  to say something beyond \eqref{Eq. Marstrand}. 

 Indeed, for an integer $2\leq m\in \mathbb{N}$, define the $m$-fold map of the unit interval 
 \begin{equation} \label{Eq T_m}
 T_m : [0,1] \rightarrow [0,1), \quad T_m(x) = m\cdot x \mod 1.
 \end{equation}
When we say that a line is not principal, we mean that its slope is in $\mathbb{R}\setminus \lbrace 0 \rbrace$, i.e. it is not parallel to the principal axes of $\mathbb{R}^2$. The following Conjecture, known also as Furstenberg's slicing Conjecture, is an example of the heuristic described in the previous paragraph:
 \begin{Conjecture} (Furstenberg, \cite{furstenberg1970intersections}) \label{Conjecutre Furstenberg}
Let $\emptyset \neq X,Y \subseteq [0,1]$ be closed sets that are invariant under $T_m$ and $T_n$, respectively.  If $\frac{\log n}{\log m} \notin \mathbb{Q}$ then for every non principal line $\ell$,
\begin{equation*}
\dim_H \ell \cap \left( X \times Y \right) \leq \max \lbrace \dim_H X +\dim_H Y -1 ,0 \rbrace.
\end{equation*}
 \end{Conjecture}
 
 Recently, two landmark papers have proven, simultanesouly and independently, this conjecture  to be correct: One of them, by Shmerkin \cite{shmerkin2016furstenberg}, proved it by computing the $L^q$ dimensions of all the projections of products of  invariant  Cantor-Lebesgue measures. The second approach, by Wu \cite{wu2016proof}, followed initially along the original idea of Furstenberg by constructing a stationary distribution on the space of measures on slices of $X\times Y$ (a CP distribution). Wu then applied Sinai's factor Theorem, "forcing" many slices of large dimension to pass through a small region in the unit square, which yielded the Conjecture. In this paper, we shall take after Wu's approach.
 
 The objectives of this paper are threefold. The first is to generalize the phenomenon predicted by Conjecture \ref{Conjecutre Furstenberg} to more general product sets, and in particular, to products of sets that are not necessarily $T_m$ invariant for some $m$ (The results of Shmerkin and Wu do not apply for these sets).  The second objective is to apply these results in order to prove slicing Theorems for Bedford-McMullen carpets with independent exponents. The third objective is to apply the results on slicing Theorems for product sets in order to prove some rigidity results in the class of Bedford-McMullen carpets. Namely, for two carpets that are incommensurable in a sense that will be defined below (and satisfy some other varying conditions), we bound non-trivially the  dimension of their intersection, show that a large class of self affine measures  on them are mutually singular, and show that they do not embed affinely into one another.
 
In the subsequent Section we outline our results in the context of the latter two objectives, which form the main results of this paper. The Section following it outlines our results in the context of the first objective, which forms our main technical tool.

\subsection{Main results} \label{Section appl}
Our main results are about geometric properties of Bedford-McMullen carpets.  These are defined as follows: let $m\neq n$ be integers greater than one, and denote, for every integer, $[n]:=\lbrace 0,...,n-1\rbrace$. We shall always assume $m>n$. Let 
\begin{equation*}
\Gamma \subseteq \lbrace 0,...,m-1 \rbrace \times \lbrace 0,...,n-1 \rbrace \;=\;[m]\times[n],
\end{equation*}
and define
\begin{equation*}
F = \lbrace (\sum_{k=1} ^\infty \frac{x_k}{m^k}, \sum_{k=1} ^\infty \frac{y_k}{n^k}) :\quad  (x_k,y_k) \in \Gamma \rbrace.
\end{equation*}
$F$ is then called a Bedford-McMullen carpet with defining exponents $m,n$, and allowed digit set $\Gamma$. For every $j\in [n]$ let
\begin{equation} \label{Eq motivation}
\Gamma_j := \lbrace i \in [m] :\quad  (i,j)\in \Gamma \rbrace \subseteq [m].
\end{equation}
We shall always assume that our carpets do not lie on a single vertical or horizontal line. When we have two carpets $F$ and $E$ we shall denote the set of allowed digits of $E$ by $\Lambda$.
\subsubsection{Dimension of slices through Bedford-McMullen carpets}
We denote by $P_2 :\mathbb{R}^2 \rightarrow \mathbb{R}$ the principal projection $P_2(x,y)=y$. We shall use the same notation for the coordinate projection in $([m]\times [n])^\mathbb{N}$.
\begin{theorem} \label{Theorem star dimension}
Let $F$ be a Bedford-McMullen carpet with exponents $(m,n)$ such that $\frac{\log m}{\log n}\notin \mathbb{Q}$. Let $\ell$ be any non-principal line in the plane. Then 
\begin{equation*}
\overline{\dim}_B (\ell \cap F) \leq  \max \lbrace \dim_H P_2 (F) + \max_{i\in [n]} \frac{\log |\Gamma_i|}{\log m} -1, 0 \rbrace
\end{equation*}
\end{theorem}

The bound obtained in Theorem \ref{Theorem star dimension} comes from the star dimension of the carpet $F$, a notion introduced by Furstenberg in \cite{furstenberg2008ergodic}: For any set $A$ we define
\begin{equation} \label{Eq star dim}
\dim^* A:= \sup \lbrace \dim_H M: \quad M  \text{ is a microset of } A \rbrace
\end{equation}
where microsets of $A$ are limits in the Hausdorff metric of "blow-up" of increasingly small balls about points in $A$ (for a formal definition of a microset, and some discussion of them, see Section \ref{Section microsets}). Now,  in \cite{mackay2011assouad}, Mackay proved that for a Bedford-McMullen carpet $F$
\begin{equation} \label{Eq star dim of carpet}
\dim^* F = \dim_H P_2 (F) + \max_{i\in [n]} \frac{\log |\Gamma_i|}{\log m}.
\end{equation}
Thus, Theorem \ref{Theorem star dimension} implies that $\overline{\dim}_B (\ell \cap F) \leq  \max \lbrace \dim^* F-1,0 \rbrace$ for any non-principal line $\ell$. 

Also, notice that if for every $i\neq j \in P_2 (\Gamma)$ we have $|\Gamma_i|=|\Gamma_j|$ then it is known that $\dim_H F = \dim^* F$ (this follows from the original works of McMullen \cite{mcmullen1984hausdorff} and Bedford \cite{bedford1984crinkly}, see also a proof in \cite{bishop2013fractal}). Therefore, in this situation, we recover the "optimal" bound, in the sense of \eqref{Eq. Marstrand} and Conjecture \ref{Conjecutre Furstenberg}. However, in general $\dim_H F \lneq \dim^* F$, and we do not know weather Theorem \ref{Theorem star dimension} can be optimized to give that $\dim_H F -1$ bounds the dimension of any non principal slice.

\subsubsection{Rigidity phenomena in the class of Bedford-McMullen carpets}
Let $F$ and $E$ be two Bedford McMullen carpets with defining exponents $(m_1,n_1)$ and $(m_2, n_2)$ respectively, and allowed digits sets $\Gamma$ and $\Lambda$.
\begin{Definition}
We shall say that  $F$ and $E$ are incommensurable if
\begin{equation*}
\frac{\log m_1}{\log m_2}, \quad \frac{\log m_1}{\log n_2}, \quad \frac{\log n_1}{\log m_2}, \quad \frac{\log n_1}{\log n_2},
\end{equation*}
are all not in $\mathbb{Q}$.
\end{Definition}

In this section we shall describe several results about geometric rigidity of incommensurable Bedford-McMullen carpets. The following result gives a bound on the dimension of intersections of such carpets. When we write $\dim$ we always mean Hausdorff dimension.
\begin{theorem} \label{Theorem intersections of carpets}
Let $F$ and $E$ be two incommensurable Bedford-McMullen carpets. Let $g:\mathbb{R}^2 \rightarrow \mathbb{R}^2$ be an affine map.
\begin{enumerate}
\item If the linear part of $g$ is given by a diagonal matrix then
\begin{equation*}
\dim^* (g(F)\cap E)\leq  \max_{(i,j)\in [n_1]\times [n_2]} \lbrace  \frac{\log |\Gamma_i|}{\log m_1} + \frac{\log |\Lambda_j|}{\log m_2} -1 ,0 \rbrace +\max \lbrace \dim P_2 (F) + \dim P_2 (E) -1,0 \rbrace.
\end{equation*}

\item If the linear part of $g$ is given by an anti-diagonal matrix then
\begin{equation*}
\dim^* (g(F)\cap E)\leq  \max_{i\in [n_1]} \lbrace  \frac{\log |\Gamma_i|}{\log m_1} + \dim P_2 (E) -1 ,0 \rbrace +\max_{j\in [n_2]}  \lbrace \dim P_2 (F) + \frac{\log |\Lambda_j|}{\log m_1} -1,0 \rbrace.
\end{equation*}
\end{enumerate} 
\end{theorem}
Theorem \ref{Theorem intersections of carpets} is related to a long line of research about intersections of Cantor sets. Notable realted works include, for example, those of Shmerkin \cite{shmerkin2016furstenberg} and Wu \cite{wu2016proof} that proved Conjecture \ref{Conjecutre Furstenberg}, the work of Feng, Huang and Rao \cite{feng2014affine}, and the work of Elekes, Keleti and M\'{a}th\'{e} \cite{elekes2010self}. Also, it is quite easy to see that the assumption that the carpets are incommensurable cannot be lifted from Theorem \ref{Theorem intersections of carpets}.

Next, we discuss self affine measures on Bedford-McMullen carpets. First, we define maps $\pi_{m_1} : [m_1]^\mathbb{N} \rightarrow [0,1]$ and $\pi_{m_2} : [m_2]^\mathbb{N} \rightarrow [0,1]$ by
\begin{equation} \label{Eq. projection from symbolic}
\pi_{m_1} (\xi) = \sum_{i=1} ^\infty \frac{\xi_i}{m_1 ^i}, \quad \pi_{m_2} (\zeta) = \sum_{i=1} ^\infty \frac{\zeta_i}{m_2 ^i}, \quad (\xi,\zeta) \in [m_1]^\mathbb{N} \times [m_2]^\mathbb{N}.
\end{equation}
A self affine measure $\mu$ on a Bedford-McMullen carpet $F$ is the push-forward $\pi_{m_1} \times \pi_{n_1} (\nu)$  of a Bernoulli measure $\nu \in P( \Gamma^\mathbb{N})$ (i.e. a stationary product measure), where $P(X)$ denotes the probability measures on a Borel space $X$. 

\begin{theorem} \label{Theorem mutual singularity}
Let $F$ and $E$ be two incommensurable Bedford-McMullen carpets, and let $\mu \in P(F)$ and $\nu\in P(E)$ be two self affine measures. Let $\kappa := \max \lbrace \dim_H \mu, \dim_H \nu  \rbrace$.  

\begin{enumerate}
\item If 
\begin{equation*} 
\kappa  \gneqq \max_{(i,j)\in [n_1] \times [n_2]} \lbrace  \frac{\log |\Gamma_i|}{\log m_1} +\frac{\log |\Lambda_j|}{\log m_2} - 1, 0 \rbrace + \max \lbrace \dim P_2 (F) + \dim P_2 (E)-1,0\rbrace
\end{equation*}
Then for any affine map $g:\mathbb{R}^2 \rightarrow \mathbb{R}^2$ such that the linear part of $g$ is a diagonal matrix, the measures $g\mu$ and $\nu$ are mutually singular . 

\item If
\begin{equation*}
\kappa  \gneqq  \max_{i\in [n_1]} \lbrace  \frac{\log |\Gamma_i|}{\log m_1} + \dim P_2 (E) -1 ,0 \rbrace +\max_{j\in [n_2]}  \lbrace  \frac{\log |\Lambda_j|}{\log m_2}+\dim P_2 (F) -1,0 \rbrace
\end{equation*}
Then for any affine map $g:\mathbb{R}^2 \rightarrow \mathbb{R}^2$ such that the linear part of $g$ is an  anti-diagonal matrix, the measures $g\mu$ and $\nu$ are mutually singular 
\end{enumerate}
\end{theorem}

For the definition of the dimension of a measure, we refer the reader to Section \ref{Section dimension}. Theorem \ref{Theorem mutual singularity} is an analogue in higher dimension of a Theorem of Hochman (\cite{hochman2010geometric}, Theorem 1.4). By this Theorem, if $\frac{\log m}{\log n}\notin \mathbb{Q}$ then  any diffeomorhic image of an ergodic $T_m$ invariant measure on $\mathbb{R}/\mathbb{Z}$, and any ergodic $T_n$ invariant measure on $\mathbb{R}/\mathbb{Z}$ are mutually singular, assuming both have intermediate dimension (recall the definition of the $m$-fold map of  the interval  $T_m$ from \eqref{Eq T_m}).

Finally, we discuss affine embeddings of incommensurable Bedford-McMullen carpets. Let $F$ and $E$ be two Bedford-McMullen carpets. We say that $F$ may be affinely embedded into $E$ if there exists an invertible affine map $g:\mathbb{R}^2 \rightarrow \mathbb{R}^2$ such that $g(F)\subseteq E$. 
\begin{theorem} \label{Theorem embeddings}
Let $F$ and $E$ be two incommensurable Bedford-McMullen carpets. Assume that $\min_{i\in [n_1]} |\Gamma_i|>1$, and that $\dim^* E <2$. Then $F$ does not admit an affine embedding into $E$.
\end{theorem}

Theorem \ref{Theorem embeddings} is related to the recently developed theory of affine embeddings of Cantor sets. The first to study such problems (for self similar sets) were Feng, Huang and Rao in \cite{feng2014affine}. In the same paper they formulated a Conjecture, stating that if one self similar set embeds into the other, then every one of its contraction ratios should be algebraically dependent on the contractions of the other set. This Conjecture was resolved for homogeneous self similar sets in dimension one by Shmerkin and Wu, in the papers proving Conjecture \ref{Conjecutre Furstenberg}, but remains open in general (for some partial results see also \cite{feng2016affine} and \cite{algom2017affine}). There is a clear relation between this Conjecture and Theorem \ref{Theorem embeddings}: our Theorem says that if $F$ embeds into $E$ then the eigenvalues of the matrices in a generating IFS for $F$ are dependent on those of $E$, which is an analogue (in an appropriate sense) of the latter Conjecture.

Finally, we do not know weather the assumptions on the dimensions of $F$ and $E$ are a by-product of our proof, or form genuine obstructions. The assumption that the carpets are incommensurable cannot be even slightly weakened in the general case, as the following example shows. Consider the carpet $F$ defined by the exponents $(3,2)$ and the digit set 
\begin{equation*}
\Gamma = \lbrace (0,0), (0,1), (2,0) \rbrace
\end{equation*}
and let $E$ be the carpet defined by exponents $(5,3)$ and the digit set
\begin{equation*}
\Lambda = \lbrace (i,0),(j,2),(1,1):0\leq i,j\leq 4 \rbrace.
\end{equation*}
Notice that $\dim^* E =2$. Then, although $\frac{\log 2}{\log 5}\notin \mathbb{Q}$, it is not hard to see that we have 
\begin{equation*}
\begin{pmatrix}
0 & 1 \\
1 & 0 \\
\end{pmatrix} \cdot F \subset E.
\end{equation*}

\subsection{A slicing Theorem for products of Cantor sets}
We obtain the results of Section \ref{Section appl}  as applications of the following slicing Theorem. Let us first describe its setup. Let $m_1 > m_2\geq 2$ and $n_1, n_2\geq 2$ be integers. Unless stated otherwise, we always assume $\theta :=\frac{\log m_2}{\log m_1} \notin \mathbb{Q}$. For every $i\in [n_1]$ we associate a subset $\emptyset\neq  \Gamma_i \subseteq [m_1]$, and for every $j\in [n_2]$ we associate a subset $\emptyset \neq \Lambda_j \subseteq [m_2]$. We always assume that there exists some $i\in [n_1]$ such that $\Gamma_i \neq [m_1]$, and similarly a $j\in [n_2]$ such that $\Lambda_j \neq  [m_2]$. Our setup (and notation) are motivated by Bedford-McMullen carpets, and the notation we have used for them in Section \ref{Section appl}, in particular \eqref{Eq motivation}.

Thus, given $\omega \in [n_1]^\mathbb{N}$ and $\eta \in [n_2]^\mathbb{N}$ we define product sets
\begin{equation} \label{Eq symbolic slice}
\tilde{F}_\omega = \prod_{i=1} ^\infty \Gamma_{\omega_i} \subseteq [m_1]^\mathbb{N}, \quad \tilde{E}_\eta = \prod_{i=1} ^\infty \Lambda_{\eta_i} \subseteq [m_2]^\mathbb{N}.
\end{equation}
In particular, for $\omega \in [n_1]^\mathbb{N}$ and $\eta \in [n_2]^\mathbb{N}$ we have
\begin{equation*}
\Fomega = \lbrace \sum_{i=1} ^\infty \frac{x_i}{m_1 ^i} :\quad  x_i \in \Gamma_{\omega_i} \rbrace ,\quad  \Eeta = \lbrace \sum_{i=1} ^\infty \frac{y_i}{m_2 ^i} : \quad  y_i \in \Lambda_{\eta_i} \rbrace,
\end{equation*}
where the maps $\pi_{m_i}$ were defined in \eqref{Eq. projection from symbolic}.

\begin{theorem} \label{Theorem main Theorem}
\begin{enumerate}
\item Let $\ell \subset \mathbb{R}^2$ be a non-principal line, and let $(\omega,\eta)\in [n_1]^\mathbb{N} \times [n_2]^\mathbb{N}$. Then
\begin{equation*}
\overline{\dim}_B \left (\Fomega \times \Eeta\right) \cap \ell \leq  \max_{i\in [n_1],j\in [n_2]} \lbrace \frac{\log |\Gamma_i|}{\log m_1}+ \frac{\log |\Lambda_j|}{\log m_2} -1, 0 \rbrace.
\end{equation*}

\item Let $u\in \mathbb{R}\setminus \lbrace 0 \rbrace$, and let $\alpha_1$ and $\alpha_2$ be Bernoulli measures on  $[n_1]^\mathbb{N}$ and  $[n_2]^\mathbb{N}$, respectively. Then there exists a measurable set $A(u,\alpha_1,\alpha_2)\subseteq [n_1]^\mathbb{N} \times [n_2]^\mathbb{N}$ of full $\alpha_1\times \alpha_2$ measure such that for all $(\omega,\eta)\in A(u,\alpha_1,\alpha_2)$, and for any line $\ell$ with slope $u$,
\begin{eqnarray*}
  \overline{\dim}_B \left (\Fomega \times \Eeta\right) \cap \ell &\leq & \max \lbrace \dim_H \Fomega + \dim_H \Eeta -1, 0 \rbrace 
   \end{eqnarray*}
\end{enumerate}
\end{theorem}

If for all $i \neq j\in [n_1]$ we have $|\Gamma_i| = |\Gamma_j|$ then $\dim_H \Fomega  = \frac{\log |\Gamma_i|}{\log m_1}$, and similarly for $\Eeta$ if $|\Lambda_j|$ is constant for all $j\in [n_2]$. Thus, by Theorem \ref{Theorem main Theorem} part (1), we recover many new explicit  examples of product sets satisfying the Furstenberg slicing bound, in the sense of Conjecture \ref{Conjecutre Furstenberg}. Moreover, by this observation and an approximation argument, it is possible to show that Theorem \ref{Theorem main Theorem} implies Conjecture \ref{Conjecutre Furstenberg}. However, as our method is based on Wu's method from \cite{wu2016proof}, this does not yield a new proof.

\subsection{On the proof of Theorem \ref{Theorem main Theorem}}
Let $m_1>m_2 \geq 2$ be integers such that $\theta:= \frac{\log m_2}{\log m_1} \notin \mathbb{Q}$. First, let $\emptyset \neq X,Y \subseteq [0,1]$ be two closed sets that are $T_{m_1}$ and $T_{m_2}$ invariant sets, respectively.  Let $\ell\cap (X\times Y)$ be any non princpal slice thorugh the corresponding product set. In \cite{wu2016proof}, Wu proved $\dim \ell\cap (X\times Y) \leq \max \lbrace \dim X +\dim Y-1,0\rbrace$ (and thus Conjecture \ref{Conjecutre Furstenberg}) by first  constructing a well structured measure (a CP distribution) on the space of measures on slices of $X\times Y$. Two key features of this measure are that its marginal on the slopes of these slices is the Lebesgue measure, and that almost all of these slices have at least the same dimension as the original slice $\ell\cap (X\times Y)$ . The construction of such a measure, originally due to Furstenberg in \cite{furstenberg1970intersections}, relies on the following observation: For every $t\in \mathbb{T}:= \mathbb{R}/\mathbb{Z}$, define a map $\Phi_t : [0,1]^2 \rightarrow [0,1]^2$, by
$$\Phi_t (z) = 
     \begin{cases}
      (T_{m_1} (z_1), T_{m_2}  (z_2))   &\quad\text{if } t\in [1 - \theta,1) \\
       ( z_1, T_{m_2} ( z_2))  &\quad\text{if } t\in [0, 1-\theta) \\
     
     \end{cases}$$
Notice that, if $m_1 ^t$ is the slope of $\ell$, then the map $\Phi_t$ transforms our slice into a finite family of slices through $X \times Y$, such that their slope corresponds to the translation by $\theta$ in $\mathbb{T}$ of $t$, and at least one has the same dimension as the original slice. 

 Wu then proceeded to apply Sinai's factor Theorem, allowing him to show that many slices that are both of dimension at least $\dim \ell\cap (X\times Y)$, and such that their slopes correspond to  sets of arbitrarily large density in an equidistributed sequence in $\mathbb{T}$,  pass through a small region in the unit square. This yielded the desired bound on $\dim \ell\cap (X\times Y)$  by a Fubini type argument. 

We take a similar approach, but we construct our CP distribution on a larger parameter space:  The space of non-principal slices of all product sets in the family $$\lbrace \Fomega\times \Eeta: \quad (\omega,\eta)\in [n_1]^\mathbb{N} \times [n_2]^\mathbb{N} \rbrace.$$ We also define, for $t\in \mathbb{T}$,  a map $\sigma_t : [n_1]^\mathbb{N} \rightarrow [n_1]^\mathbb{N}$ by 
$$\sigma_t (\omega) = 
     \begin{cases}
      \sigma (\omega)  &\quad\text{if } t \in [1- \theta,1) \\
       \omega &\quad\text{if } t\in [0, 1-\theta) \\
     
 \end{cases}$$
   The basic observation behind our approach is that now, for any non principal slice $\ell \cap(\Fomega\times \Eeta)$ through any product set in our family, the map $\Phi_t$ transforms this  slice  into a finite family of slices through $\pi_{m_1} (\tilde{F}_{\sigma_t (\omega)}) \times \pi_{m_2} (\tilde{E}_{\sigma (\eta)})$, where $\sigma:[n_i]^\mathbb{N}\rightarrow [n_i]^\mathbb{N}$ is the left shift. It is still true that their slopes correspond to the original slope translated by $\theta$ in $\mathbb{T}$, and at least one has the same dimension as the original slice. Notice that this is a slice through (possibly a different) product set in our family.
   
An application of Sinai's factor Theorem yields a similar conclusion to that of Wu's,  that many slices in this family that are both of dimension at least $\dim \ell \cap(\Fomega\times \Eeta)$,  and such that their slopes correspond to sets of arbitrarily large density in an equidistributed sequence in $\mathbb{T}$,  pass through a small region in the unit square. Moreover, using this idea we can also show that the amount of product sets in our family being sliced in this procedure is not too large (in some sense), allowing for a Fubini argument (similar, but more complicated, than that of Wu's), to be preformed.

However, unless we have some additional information about the $(\omega,\eta)$ from Theorem \ref{Theorem main Theorem} part (1) (as we do in Theorem \ref{Theorem main Theorem} part (2)), we cannot control which product sets will play a part in the end game of this procedure. This explains the bound appearing in part (1) of the Theorem (which is the "worst case scenario" - the largest possible box dimension of a product set in our family).

\textbf{Notation} This paper is particularly related to the work of Wu \cite{wu2016proof}, and to our previous work with Hochman \cite{algom2016self} (via Theorems \ref{Theorem - structure of tangnet sets} and \ref{Theorem - structure of tangent sets general} in Section \ref{Section micro of BMC} below). Thus, we make an effort to use similar notation as both of these works. Otherwise, we use standard notation: For example, Greek letters shall usually denote measures (the maps defined in \eqref{Eq. projection from symbolic}, which are defined as in \cite{algom2016self}, are one exception to this rule), lower case Latin letters denote maps, and upper case Latin letters shall denote sets. 
 
\textbf{Organization} In section \ref{Section prel.} we survey some relevant definitions and results  about dimension theory of sets and measures, and about CP distributions. We then proceed to prove, in section \ref{Section app}, Theorems \ref{Theorem star dimension}, \ref{Theorem intersections of carpets}, \ref{Theorem mutual singularity}, \ref{Theorem embeddings}, assuming Theorem \ref{Theorem main Theorem} is correct. The subsequent sections are then devoted to the proof of Theorem \ref{Theorem main Theorem}, and related constructions.

\textbf{Acknowledgements} This work was carried out as part of the author's research towards a PhD dissertation, conducted at the Hebrew University of Jerusalem. I would like to thank my advisor, Michael Hochman, for many helpful discussions and useful suggestions. I would also like  to thank Pablo Shmerkin and Meng Wu for some interesting discussions related to the topic of this paper.

\section{Preliminaries} \label{Section prel.}
Let $X$ be a metric space. The set of Borel probability measures on $X$ will be denoted by $P(X)$. In this paper, all measures are  Borel probability measures.
     
 \subsection{Some notions of dimension of sets and measures} \label{Section dimension}
For a set $A$ in some metric space, we use the standard notation $\dim_H A$ for the Hausdorff dimension of $A$, and $\overline{\dim}_B (A)$ for the upper box dimension of $A$. See e.g. Falconer's book \cite{falconer1986geometry} for some exposition on these concepts. 

Next, let $\mu$ be a Borel probability measure on some metric space. For every $x\in \supp(\mu)$  we define the  pointwise (exact) dimension of $\mu$ at $x$  as
\begin{equation*}
\dim(\mu,x)=\lim_{r\rightarrow 0} \frac{\log \mu (B(x,r))}{\log r}
\end{equation*}
where $B(x,r)$ denotes the closed ball or radius $r$ about $x$. If the limit does not exist, we define the upper and lower pointwise dimensions of $\mu$ at $x$ as the corresponding $\limsup$ and $\liminf$.

We also define the (lower) Hausdorff dimension of the measure $\mu$ as
\begin{equation*}
\dim (\mu) = \inf \lbrace \dim_H A : \quad \mu (A)>0 \rbrace.
\end{equation*}
If the pointwise dimension of $\mu$ exists at almost every $x\in \supp(\mu)$ and is constant almost surely, then this constant value is known to equal $\dim (\mu)$. For proofs and some more discussion, see e.g. \cite{falconer1997techniques} or \cite{Fan2002measures}.

Next, we discuss entropy of measures and entropy dimension. First, let $\mu$ be a Borel probability measure on some metric space. Let $\mathcal{A}$ denote a countable (or finite) partition of the underlying space. Then the entropy of $\mu$ with respect to $\mathcal{A}$ is defined as
\begin{equation*}
H(\mu,\mathcal{A}) = -\sum_{A\in \mathcal{A}} \mu(A)\cdot \log \mu (A)
\end{equation*}
with the convention  $0\log 0=0$.

Let us now define the entropy dimension of a measure $\mu \in P(\mathbb{R}^d)$. For every integer $p\geq 0$ let $\mathcal{D}_p$ denote the $p$-adic partition of $\mathbb{R}^d$, that is,
\begin{equation*}
\mathcal{D}_p = \lbrace \prod_{i=1} ^d [\frac{z_i}{p}, \frac{z_i+1}{p}) :\quad  (z_1,...,z_d)\in \mathbb{Z}^d \rbrace.
\end{equation*}
The entropy dimension of $\mu$ is defined as
\begin{equation*}
\dim_e (\mu) = \lim_{k\rightarrow \infty} \frac{1}{k\log 2}H(\mu,\mathcal{D}_{2^k}),
\end{equation*}
provided that the limit exists. If the limits does not exist, the upper and lower entropy dimension of $\mu$ are defined as the corresponding $\limsup$ and $\liminf$.

Next, let $n\geq 2$ and consider the symbolic space $[n]^\mathbb{N}$, with the usual product topology. For every finite word $u\in [n]^k$ for some $k\in \mathbb{N}$ we associate its length, defined by $|u|=k$, and a cylinder set defined by
\begin{equation*}
[u]:=\lbrace \omega \in [n]^\mathbb{N}: \quad (\omega_1,...,\omega_k)=u \rbrace.
\end{equation*} 
Though this coincides with the notation $[n]$, which notion is meant will be clear from context. Let $\mathcal{I}_k$ denote the partition of $[n]^\mathbb{N}$ into cylinders of length $k$. For a measure $\mu \in P([n]^\mathbb{N})$ we define the entropy dimension of $\mu$ as
\begin{equation*}
\dim_e (\mu) = \lim_{k\rightarrow \infty} \frac{1}{k\log 2}H(\mu,\mathcal{I}_k),
\end{equation*}
provided that the limit exists. If the limits does not exist, the upper and lower entropy dimension of $\mu$ are defined as the corresponding $\limsup$ and $\liminf$.

Finally, let $\mu\in P(\mathbb{R}^d)$ or $\mu \in P( [n]^\mathbb{N} )$ for some $n\geq 2$. Then, if $\mu$ is supported on a set $A$
\begin{equation} \label{Eq entr. dim vs Hausd. dim}
\dim(\mu)\leq \underline{\dim}_e (\mu) \leq  \overline{\dim}_e (\mu) \leq  \overline{\dim}_B A.
\end{equation}
If $\mu$ is exact dimensional then
\begin{equation*}
\dim (\mu) = \dim_e (\mu).
\end{equation*}
For proofs, and more discussion of these concepts see \cite{Fan2002measures} and \cite{mattila1999geometry}.

 \subsection{Star dimension, microsets and covariance of microsets}  \label{Section microsets}  
Let $X$ be a compact metric space, which in practice will be either $[-1,1]^2$ or a symbolic spaces of the form $[n]^\mathbb{N}$. If $X=[-1,1]^2$ we shall use the Euclidean norm $|| \cdot ||$, and in the space $[n]^\mathbb{N}$ we  consider, for some $\rho\in (0,1)$, the metric $d_\rho$ on $[n]^\mathbb{N}$, defined  by
\begin{equation} \label{The metric drho}
d_\rho (x,y) = \rho^{\min\lbrace k : \quad  x_k \neq y_k \rbrace}.
\end{equation} 
Let $\cmpct(X)$ denote the set of non-empty closed subsets of $X$. For $A,B\in\cmpct(X)$ and $\epsilon >0$ define
\begin{equation*}
A_\epsilon = \lbrace x \in X : \quad \exists a\in A, d(x,a) < \epsilon \rbrace.
\end{equation*}
The Hausdorff distance between $A$ and $B$ is defined by
\begin{equation*}
d_H (A,B) = \inf \lbrace \epsilon >0 : \quad A \subseteq B_\epsilon, \quad B \subseteq A_\epsilon \rbrace. 
\end{equation*}
This is a compact metric on  $\cmpct(X)$ (see e.g. the appendix in \cite{bishop2013fractal}). 

Now,  let us restrict to $X=[-1,1]^2$. Let $F\subseteq [0,1]^2$  be a compact set. A set $A$ such that $A \subseteq [-1,1]^2$ is called a  miniset of $F$ if $A \subseteq (a \cdot F + t)\cap [-1,1]^2$ for some $a \geq 1, t\in \mathbb{R}$. A set $M$ is called a  microset of $F$ if $M$ is a limit in the  Hausdorff metric on subsets of $[-1,1]^2$ of minisets of $F$. Let $\mathcal{G}_F$ denote the family of all microsetes of $F$. Recall, from \eqref{Eq star dim}, that the star dimension of $F$ is the defined as
\begin{equation*}
\dim^* F = \sup \lbrace \dim_H A : \quad A\in \mathcal{G}_F \rbrace.
\end{equation*}
It is known that this supremum is in fact a maximum, obtained by the dimension of a limit of non-degenerate minisets, i.e. minisets of the form $(a_k \cdot F + t_k)\cap [-1,1]^2$ such that $a_k \rightarrow \infty$. For a proof, see Lemma 2.4.4 in \cite{bishop2013fractal}.

We shall also consider a special type of minisets and micorsets. Let $m>1$ and fix $x\in F$. An $m$-adic mini-set of $F$ about $x$ is a set of the form 
\begin{equation} \label{Def M-adic microset}
[m^k (F-x)]\cap [-1,1]^2 \in \cmpct([-1,1]^2), \text{ where } k\in \mathbb{N}.
\end{equation}
An $m$-adic microset of $F$ about $x$ is a limit of such sets as $k\rightarrow \infty$ (there is always a converging subsequence by the compactness of $\cmpct([-1,1]^2)$).

One of the many reasons it is interesting to study microsets is their nice behaviour with respect to affine (and more generally, smooth) embeddings. Namely, an affine embedding of one set into another set induces a corresponding affine embedding of their microsets:
\begin{Proposition} \label{Proposition Covariance}
Let $m_1,m_2 >1$ be integers, and let $g(x)=Ax+t$ be an invertible affine map of $\mathbb{R}^2$ such that $c \cdot A([-1,1]^2) \subseteq [-1,1]^2$ for all $c\in [1,m_1]$. Let $F,E\subseteq [-1,1]^2$ be compact, and suppose that $g(F)\subseteq E$. Let $x\in F$ and set $y=g(x)\in E$. Suppose that for some sequence $\lbrace n_k \rbrace \subseteq \mathbb{N}$
\begin{equation*}
\lim_{k\rightarrow \infty} [m_1 ^{[n_k \cdot \frac{\log m_2}{\log m_1}]} (F-x)]\cap [-1,1]^2 = T ,\quad \text{ and } \lim_{k\rightarrow \infty} [m_1 ^{n_k} (E-y)]\cap [-1,1]^2 = T'
\end{equation*}
Then there exists some constant $c\in [1,m_1]$ such that $c\cdot A(T) \subseteq T'$.
\end{Proposition}  
We refer to this phenomenon as "covariance of microsets". We omit the proof, since it is rather similar to the proof of Proposition 4.3 in \cite{algom2016self}. The assumption $c \cdot A([-1,1]^2) \subseteq [-1,1]^2$ for all $c\in [1,m_1]$ is needed for certain algebraic manipulations to work out. Without it, we can obtain a similar result of the form $A(T)\subseteq T'\cdot c'$, and $c'$ can be bounded in terms of the operator norm of the matrix $A$.

 \subsection{CP distributions}
 \subsubsection{Dynamical systems} \label{Section Ber mea.}
 In this paper, a measure preserving system is a quadruple $(X, \mathcal{B},T,\mu)$, where $X$ is a compact metric space, $\mathcal{B}$ is the Borel sigma algebra, $T:X\rightarrow X$ is a  measure preserving map, i.e. $T$ is Borel measurable and $T\mu = \mu$. Since we always work with the Borel sigma-algebra, we shall usually just write $(X,T,\mu)$.
 
A class of examples of a dynamical systems are symbolic dynamical systems: We take $X=[n]^\mathbb{N}$ for some $n$, we take $T=\sigma$ to be the shift map $\sigma:[n]^\mathbb{N}\rightarrow [n]^\mathbb{N}$ defined by $\sigma(\omega) = \xi$ where $\xi (k) =\omega(k+1)$ for every $k$. A special case is when  $\mu$ is a Bernoulli measure: that is, $\mu = p^\mathbb{N}$ where $p$ is probability vector $p\in P([n])$. These systems are also called Bernoulli shifts.
 
A dynamical system is ergodic if and only if the only invariant sets are trivial. That is, if $B\in \mathcal{B}$ satisfies $T^{-1} (B) = B$ then $\mu(B)=0$ or $\mu(B)=1$. A dynamical system is called weakly mixing if for any ergodic dynamical system $(Y,S,\nu)$, the product system $(X\times Y, T\times S, \mu\times \nu)$ is also ergodic. In particular, weakly mixing systems are ergodic. Moreover, If both $(X,T,\mu)$ and $(Y,S,\nu)$ are weakly mixing, then their product system is also weakly mixing. A class of examples of  weakly mixing systems is given by Bernoulli shifts.

A useful tool that will appear frequently in this paper is the ergodic decomposition Theorem:
\begin{theorem}
Let $(X,T,\mu)$ be a dynamical system. Then there is a  map $X\rightarrow P(X)$, denoted by $\mu \mapsto \mu^{x}$, such that:
\begin{enumerate}
\item The map $x\mapsto \mu^x$ is measurable with respect to the sub-sigma algebra $\mathcal{E}$ of $T$ invariant sets.

\item $\mu = \int \mu^x d\mu(x)$

\item For $\mu$ almost every $x$, $\mu^x$ is $T$ invariant, ergodic, and supported on the atom of $\mathcal{E}$ that contains $x$. The measure $\mu^x$ is called the ergodic component of $x$.
\end{enumerate}
\end{theorem}

Another useful notion is that of generic points in a dynamical system $(X,T,\mu)$. We say that a point $x\in X$ is generic with respect to $\mu$ if
\begin{equation*}
\frac{1}{N} \sum_{i=0} ^{N-1} \delta_{T^i x} \rightarrow \mu,\quad \text{ where } \delta_y \text{ is the dirac measure on } y\in X,
\end{equation*}
in the weak-* topology. By the ergodic Theorem, if $\mu$ is ergodic then $\mu$ a.e. $x$ is generic for $\mu$.

Finally, we discuss generators.  Let $\mathcal{A}$ be a finite partition of $X$. Let $\mathcal{A}_k = \bigvee_{i=0} ^{k-1} T^{-i} \mathcal{A}$ denote the coarsest common refinement of $\mathcal{A}, T^{-1} \mathcal{A}...,T^{-k+1}\mathcal{A}$. The sequence $\mathcal{A}_k$ is called the filtration generated by $\mathcal{A}$ with respect to $T$. For every $k\geq 1$ and $x\in X$, let $\mathcal{A}_k (x)$ denote the unique element of $\mathcal{A}_k$ that contains $x$. 

Now, if the smallest sigma algebra that contains $\mathcal{A}_k$ for all $k$ is the Borel sigma algebra, we say that $\mathcal{A}$ is a generator for $(X,T,\mu)$. By the Kolmogorov-Sinai Theorem, if $\mathcal{A}$ is a generator, then
\begin{equation*}
\lim_k \frac{1}{k} H(\mu,\mathcal{A}_k) = \sup_{\mathcal{B}: \mathcal{B} \text { is a countable partition of } X}  \lim_k \frac{1}{k} H(\mu,\mathcal{B}_k).
\end{equation*}  
The common value described above is called the entropy of the dynamical system $(X,T,\mu)$ and is denoted by $h(\mu,T)$.

\subsubsection{CP distributions on symbolic spaces} \label{Section CP symb}
The theory of CP distributions, that we discuss in this section, originated implicitly with Furstenberg in \cite{furstenberg1970intersections}. It was then reintroduced by Furstenberg in \cite{furstenberg2008ergodic}, and has since been used by many authors, notably by Hochman and Shmerkin in \cite{hochman2009local}. In particular, CP distributions shall play a crucial role in the proof of Theorem \ref{Theorem main Theorem}, as they do in Wu's work \cite{wu2016proof}. In this Section, we follow closely Section 3 in \cite{wu2016proof}.

As is standard in this context, if $X$ is a metric space then elements of $P(X)$ are called measures, and elements of $P(P(X))$, measures on the space of measures, are called distributions.

Let $m_1,m_2\geq 2$ and let $X =  ([m_1]\times [m_2])^\mathbb{N}$ (the theory extends to any finite alphabet, but this model will suffice for us). Fix $\rho\in (0,1)$ and consider the metric $d_\rho$ on $X$ (recall \eqref{The metric drho}). Let 
\begin{equation*}
\Omega = \lbrace (\mu,x)\in P(X)\times X : \quad x\in \supp(\mu) \rbrace.
\end{equation*}   
We define the magnification operator $M:\Omega \rightarrow \Omega$ by
\begin{equation*}
M(\mu,x) = (\mu^{[x_1]},\sigma(x))
\end{equation*}
where $[x_1] = \lbrace y\in X : y_1 = x_1 \rbrace$, and $\mu^{[x_1]} = \frac{\sigma( \mu|_{[x_1]})}{\mu([x_1])}$. 

It is clear that $M$ is continuous, and that $M(\Omega)\subseteq \Omega$. For any distribution $P\in P(\Omega)$, let $P_1$ denote its marginal on the measure coordinate. We shall say that $P$ is adapted if for every $f\in C(X)$,
\begin{equation*}
\int f(\mu,x)dP(\mu,x) = \int \left( \int f(\mu,x) d\mu(x) \right) dP_1(\mu).
\end{equation*}  
In particular, if $P$ is adapted then if a property holds $P$ almost surely, then it holds for $P_1$ almost every $\mu$, and for $\mu$ almost every $x$.

\begin{Definition} \label{Def CP dist}
A distribution $P\in P(\Omega)$ is called a CP-distribution if it is $M$ invariant and adapted. 
\end{Definition}

A CP-distribution $P$ is called ergodic if the underlying dynamical system $(\Omega,M,P)$ is ergodic. If it is not ergodic, its ergodic decomposition provides us with  ergodic CP distributions:
\begin{Proposition}  \label{Proposition components of CP}
The ergodic components of a CP-distribution are, almost surely, themselves ergodic CP-distributions.
\end{Proposition}
A proof is indicated by Furstenberg in \cite{furstenberg2008ergodic} (after Proposition 5.1), and can be deduced from Theorem 1.3 in \cite{hochman2010dynamics}.

We proceed to collect some useful properties of CP distributions. 

\begin{Proposition}  \cite{furstenberg2008ergodic}
Let $P$ be an ergodic CP-distribution. Then $P_1$ almost every measure $\mu$ is exact dimensional with dimension
\begin{equation*}
\dim \mu = \frac{1}{\log \rho^{-1}} \int -\log \nu([x_1]) dP(\nu,x).
\end{equation*}
\end{Proposition}
For an ergodic CP distribution $P$, $\dim P$ denotes this (almost surely) constant value.

Next, let $x\in X$, and denote $[x_1 ^k] = \lbrace y\in X:(y_1,..,y_k) = (x_1,...,x_k) \rbrace$. We also denote
\begin{equation*}
\mu^{[x_1 ^k]} = \frac{\sigma(\mu|_{[x_1^k]})}{\mu([x_1 ^k])}
\end{equation*}
It follows from the ergodic Theorem that if $P$ is an ergodic CP distribution, then $P_1$ almost every $\mu$ generates $P_1$ in the sense that for $\mu$ a.e. $x$
\begin{equation} \label{Equation 3.1}
\frac{1}{N}\sum_{i=1} ^{N} \delta_{\mu^{[x_1 ^i]}} \rightarrow P_1
\end{equation}
in the weak-* topology. Measures that satisfy this shall be called generic for $P_1$. 

The following Proposition was proved by Wu in \cite{wu2016proof}. We denote by $\mathcal{J}_k ^i$ the  $k$-th generation cylinder partition of $[m_i]^\mathbb{N}$ for $i=1,2$. Thus, $\mathcal{J}_k ^1 \times \mathcal{J}_k ^2$ is the $k$-th generation cylinder partition of the space $X$.
\begin{Proposition} (\cite{wu2016proof}, Proposition 3.7) \label{Proposition 3.7}
Let $P$ be an ergodic CP distribution with $\dim P =q>0$. For every $\epsilon>0$ there exists $k_0 (\epsilon) \in \mathbb{N}$ such that for each $\mu$ that is generic for $P_1$ and for $\mu$ almost every $x$,
\begin{equation*}
\liminf_{N\rightarrow \infty} \frac{1}{N} | \lbrace 1\leq p \leq N: \max_{u\in ([m_1]\times [m_2])^{k_0(\epsilon)}} \mu^{[x_1 ^p]} ([u]) \leq \epsilon \rbrace| > 1-\epsilon,
\end{equation*}
and
\begin{equation*}
\liminf \frac{1}{N} |\lbrace 1\leq p \leq N: H(\mu^{[x_1 ^p]} , \mathcal{J}_p ^1 \times \mathcal{J}_p ^2) \geq p\cdot (q\cdot \log \rho^{-1} - \epsilon)\rbrace| >1-\epsilon, \text{ for all } p\geq k_0(\epsilon).
\end{equation*}
Moreover, this is true for all pairs $(\mu,x)$ satisfying \eqref{Equation 3.1} (with $k_0$ depending on $(\mu,x)$).
\end{Proposition}
Finally, in practice we shall construct a CP distribution on a space of the form $([m_1]\times [m_2])^{\mathbb{N}_0}$, where $\mathbb{N}_0 = \mathbb{N}\cup \lbrace 0 \rbrace$. It is not hard to see how the discussion in this Section generalizes to this situation.

\subsubsection{CP distributions on Euclidean spaces}
The CP distributions discussed in the previous section have many applications for problems in geometric measure theory. To make the connection, we introduce the Euclidean version of CP distributions, which are closely related to symbolic CP distributions. In this section, we partialy follow Section 2.1 in \cite{Fraser2015Ferguson}. We introduce the theory only in $\mathbb{R}^2$, where we shall use it.

Let $B\subset \mathbb{R}^2$ be a box, that is, a product of intervals (open, closed, or half open). Let $T_B:\mathbb{R}^2 \rightarrow \mathbb{R}^2$ denote the orientation preserving affine map
\begin{equation*}
T_B (x) = \frac{1}{\sqrt{|B|}}(x-\min \overline{B}),
\end{equation*}
where $|B|$ is volume of $B$ and $\min \overline{B}$ is the minimal element of $\overline{B}$ with respect to the lexicographic order (so it's the lower left corner of the box). We define the normalized box $B^* = T_B (B)$, so that $|B^*|=1$. If $\mu \in P(\mathbb{R}^2)$ and $B$ is a box with $\mu(B)>0$ we write
\begin{equation*}
\mu^B = \frac{1}{\mu(B)} T_B (\mu|_B) \in P(B^*).
\end{equation*}

Next, we define partition operators and filtrations. Let $\mathcal{E}$ be a collection of boxes in $\mathbb{R}^2$. A partition operator $\Delta$ on $\mathcal{E}$ associates to every $B\in \mathcal{E}$ a partition $\Delta B \subset \mathcal{E}$ of $B$ such that, for every homothety $S:\mathbb{R}^2 \rightarrow \mathbb{R}^2$, we have $S(\Delta B) = \Delta(SB)$. For every $B\in \mathcal{E}$, the partition operator $\Delta$ defines a filtration of $B$ by
\begin{equation*}
\Delta^0 (B) = \lbrace B \rbrace, \quad \Delta^{n+1} (B) = \lbrace \Delta(A): \quad A\in \Delta^k (B) \rbrace.
\end{equation*}
 A partition operator $\Delta$ is called $\delta$-regular if for any $B\in \mathcal{E}$ there is a constant $c>1$ such that for all $k\in \mathbb{N}$, any element $A\in \Delta^k (B)$ contains a ball of radius $\delta^k/c$ and is contained in a ball of radius $c\delta^k$. 
 
 For example, for every $m\geq 2$ we define the base $m$ partition operator on $\mathcal{E} = \lbrace [u,v]^2 : u<v \rbrace$ by defining 
 \begin{equation*}
 \Delta_m([0,1]^2) = \lbrace [\frac{k_1}{m}, \frac{k_1+1}{m})\times [\frac{k_2}{m}, \frac{k_2+1}{m}): 0\leq k_1,k_2 < m-1, \quad k_1,k_2 \in \mathbb{Z} \rbrace,
 \end{equation*}
 and extending (by invariance) to all cubes. Notice that this operator is $\frac{1}{m}$ regular.
 
\begin{Definition} \label{Defition CP}
 Fix a collection of boxes $\mathcal{E}$ and define a state space
 \begin{equation*}
 \Theta = \lbrace (B,\mu):\quad \mu\in P(B^*),B\in \mathcal{E}\rbrace.
 \end{equation*}
 A $\delta$-regular CP-chain $Q$ with respect to a $\delta$-regular partition operator $\Delta$ is a stationary Markov process on the state space $\Theta$ with the Markov kernel 
 \begin{equation*}
 F(B,\mu) = \sum_{A\in \Delta(B^*)} \mu(A)\delta_{(A,\mu^A)},\quad (B,\mu)\in \Theta.
\end{equation*}  
 \end{Definition}
 
 Thus, by definition, if $Q$ is the unique stationary distribution with respect to the chain, then  $(\Theta^\mathbb{N},\sigma,\hat{Q})$ is a dynamical system, where $\hat{Q}$ is the extension of $Q\in P(\Theta)$ to a measure on $\Theta^\mathbb{N}$, generated by running the Markov chain starting from $Q$. We abuse notation and refer to $\hat{Q}$ as $Q$. Thus, $Q$ is ergodic if this system is ergodic. 
 
 \begin{Definition} \label{Definition CP gen}
 Let $Q$ be a CP chain as above. We abuse notation and write $Q$ for the distribution of its measure component. Given $B\in \mathcal{E}$, the CP chain $Q$ is  generated by $\mu \in P(B^*)$ if at $\mu$ almost every $x\in B^*$ 
 \begin{equation*}
 \frac{1}{N} \sum_{k=0} ^{N-1} \delta_{\mu^{\Delta^k (B) (x)}} \rightarrow Q
 \end{equation*}
 in the weak star topology, and for any $q\in \mathbb{N}$,
 \begin{equation*}
 \frac{1}{N} \sum_{k=0} ^{N-1} \delta_{\mu^{\Delta^{qk} (B) (x)}} \rightarrow Q_q
 \end{equation*}
 converge to some (possibly different) distribution $Q_q$.
  \end{Definition}

\subsubsection{Continuous time scaling scenery}
To prove Theorem \ref{Theorem mutual singularity}, we shall require the notion of the continuous scaling scenery of a measure $\mu\in P([0,1]^2)$ at a point $x\in \supp(\mu)$. First, we define the scaling and translation maps $S_t, T_x :\mathbb{R}^2 \rightarrow \mathbb{R}^2$ by
\begin{equation*}
S_t (y) = e^t \cdot y, \quad T_x (y) =y -x.
\end{equation*}
We also define the restriction and normalization operator
\begin{equation*}
\nu \in P( \mathbb{R}^2)\mapsto \nu^\square :=\left( \frac{\nu}{\nu([-1,1]^2} \right)|_{[-1,1]^2}, \text{ assuming } 0 \in \supp(\nu). 
\end{equation*}

\begin{Definition} (\cite{Gavish2008scaling}, \cite{hochman2010dynamics}) \label{Generating FD}
Let $\mu \in P([0,1]^2)$ and let $x\in \supp (\mu)$. 
\begin{enumerate}
\item  We define the parametrized family of measures $\mu_{x,t}  = \left( S_t \circ T_x (\mu) \right)^\square$. This family is called the scenery of $\mu$ at $x$. 

\item For every $T>0$ we define the scenery distribution $$\left\langle \mu \right\rangle_{x,T} = \frac{1}{T} \int_0 ^T \delta_{\mu_{x,t}} dt \in P( P([-1,1]^2)).$$

\item If $\left\langle \mu \right\rangle_{x,T} \rightarrow P$ as $T\rightarrow \infty$ we say that $\mu$ generates $P$ at $x$.
\end{enumerate}
 \end{Definition}
  
One of the main advantages of zooming into a measure in this way is that it is done in a coordinate free way. An example of how this is useful is the following Lemma:

\begin{Lemma} (\cite{hochman2010dynamics}) \label{Lemma generating w.r.t. image and abs. conti.}
Let $\mu\in P([0,1]^2)$ be a Borel probability measure such that for $\mu$ almost every $x\in [0,1]^2$, $\left\langle \mu \right\rangle_{x,T} \rightarrow P$, for some $P\in P(P([0,1]^2))$. 

\begin{enumerate}
\item If $\nu \ll \mu$ then $\nu$ generates $P$ at almost every $x$.

\item Let $g\in \text{diff} (\mathbb{R}^2)$. Then for $g\mu$ almost every $g(x)$, $\left\langle g\mu \right\rangle_{g(x),T}\rightarrow \left( Dg(x) \right)^\square P$, where $( Dg(x) )^\square$  transforms measures by first pushing them forward via $D_g(x)$ and then applying $^\square$.
\end{enumerate} 
\end{Lemma}

The following Theorem, due to Gavish \cite{Gavish2008scaling}, and in greater generality to Hochman \cite{hochman2010dynamics}, shows that a measure that generates an ergodic CP distribution also generates a distribution in the sense of Definition \ref{Generating FD}. Moreover, using the centering operation (see \cite{hochman2010dynamics}) we are able to relate the two distributions:
\begin{theorem} (\cite{Gavish2008scaling}, \cite{hochman2010dynamics})  \label{Theorem Gavish}
Let $\mu \in P([0,1]^2)$ be a measure that generates an ergodic CP distribution $Q$ in the sense of Definition \ref{Definition CP gen}. Then $\mu$ generates a distribution $P\in P(P([0,1]^2)$ at $\mu$ almost every $x$, in the sense of Definition \ref{Generating FD}. Moreover, there exists a distribution $R$ on triplets of the form $(\rho,\nu,x)$ such that:
\begin{enumerate}
\item The first coordinate $\rho$ is distributed according to $P$
.
\item The second coordinate $\nu$ is distributed according to $Q$, and $x$ is distributed according to $\nu$.

\item For $R$-almost every such triplet there exist $r(\rho,\nu,x)>0$ and $t>0$ such that $$\rho|_{B(0,r)} \ll S_t \circ T_x (\nu).$$
\end{enumerate}
\end{theorem}

 \subsection{Bedford-McMullen carpets}
 \subsubsection{Iterated function systems}
 
Let $\Phi = \lbrace \phi_i \rbrace_{k=1} ^l , l\in \mathbb{N}, l\geq 2$ be a family of contractions $\phi_i : \mathbb{R}^d \rightarrow \mathbb{R}^d, d\geq 1$. The family $\Phi$ is called an iterated function system, abbreviated IFS, the term being coined by Hutchinson \cite{hutchinson1981fractals}, who defined them and studied some of their fundamental properties. In particular, he proved that there exists a unique compact $\emptyset \neq F \subset \mathbb{R}^d$ such that $F = \bigcup_{i=1} ^l \phi_i (F)$. $F$ is called the attractor of $\Phi$, and $\Phi$ is called a generating IFS for $F$.  A set $F \subset \mathbb{R}^d$ will be called self similar if there exists a generating IFS $\Phi$ for $F$ such that $\Phi$ consists only of similarity mappings. Similarly, if $\Phi$ consists only of affine maps, then we say that $F$ is a self affine set.

The self similar sets we shall encounter in this paper are deleted digit sets: for an integer $n\geq 2$, Let $D \subseteq [n]$. Define an IFS $\Phi = \lbrace f_i \rbrace_{i\in D}$, where 
\begin{equation*}
\forall i\in D,  \forall x\in \mathbb{R}, \quad f_i (x) = \frac{x+i}{n}.
\end{equation*}
The attractor of $\Phi$ is called a deleted digit set. These sets are quite nice. For example, if $K$ is a deleted digit set then
\begin{equation*}
\dim_H K = \dim_B K = \dim^* K = \frac{\log |D|}{\log n}.
\end{equation*}

Finally, we discuss self similar measures on deleted digit sets. Let $K$ be a deleted digit set as above. A measure $\mu \in P(K)$ is called a self similar measure if there exists a fully supported Bernoulli measure $\alpha \in P(D^\mathbb{N})$ such that $\pi_n \alpha = \mu$ (recall the map $\pi_n$ from \eqref{Eq. projection from symbolic}). These measures are known to be exact dimensional (in much greater generality, see \cite{feng2009dimension}) of dimension $\dim K$.

 \subsubsection{Bedford-McMullen carpets}
We now recall some basic concepts regarding Bedford-McMullen carpets. We follow the terminoloy of \cite{algom2016self}, which motivates our notation with regard to Theorem \ref{Theorem main Theorem}. Recall the definition of a Bedford-McMullen carpet $F$ with defining exponents $m,n$ and allowed digit set $\Gamma$ from section \ref{Section appl}. Notice that if $F$ is a Bedford-McMullen carpet then both $P_1 (F), P_2 (F)$ are deleted digit sets. Also, note that $F$ is a self affine set generated by an IFS consisting of maps whose linear parts are diagonal matrices. Specifically, $F$ is the attractor of $\Phi = \lbrace \phi_{(i,j)} \rbrace_{(i,j) \in \Gamma}$ where
\begin{equation} \label{Genrating IFS for F}
\phi_{(i,j)} (x,y) = (\frac{x+i}{m}, \frac{y+j}{n}) = \begin{pmatrix}
\frac{1}{m} & 0 \\
0 & \frac{1}{n}
\end{pmatrix}
\cdot (x,y) + (\frac{i}{m},\frac{j}{n}).
\end{equation}
Recall that when we have two carpets $F$ and $E$ we shall denote the set of allowed digits of $E$ by $\Lambda$.

Recall the ``projection'' $\pi_m : [m]^\mathbb{N} \rightarrow [0,1]$ defined in \eqref{Eq. projection from symbolic}. This is a continuous surjection to $[0,1]$, but can fail to be injective on countably many points, specifically, rationals in $(0,1)$ of the form $k/m^n$ have two preimages under $\pi_m$ (but note that $0,1$ have only one pre-image). We also define, by a slight abuse of notation, the projection  $\pi_m \times \pi_n: ([m]\times [n])^\mathbb{N} \rightarrow [0,1]^2$. Then  $\widetilde{F} = \Gamma^\mathbb{N} \subseteq ([m]\times [n])^\mathbb{N}$ is a shift invariant subset satisfying $\pi_m \times \pi_n (\widetilde{F}) = F$. As before, this may not be an injection, even though it is surjective, and $\widetilde{F} \subseteq (\pi_m \times \pi_n) ^{-1} (F)$, but the two sets might not be equal.

Recall that for $y \in P_2 (F)$ we defined $F_y$ as the horizontal slice $F_y = \lbrace x \in \mathbb{R} : (x,y) \in F \rbrace$. Note  that $F_y \times \lbrace y \rbrace = F\cap (\mathbb{R} \times \lbrace y \rbrace)$. In the symbolic context, for an infinite sequence $\omega \in [n]^\mathbb{N}$ we define the symbolic slice corresponding to $\omega$ by
\begin{equation*}
\widetilde{F} _{\omega} = \lbrace  \eta \in [m]^\mathbb{N} : ( \eta , \omega ) \in \widetilde{F} \rbrace \;=\;\prod_{i=1}^\infty \Gamma_{\omega_i},
\end{equation*}
where for $i\in [n]$, $\Gamma_i$ was defined in section \ref{Section appl}. Notice that this coincides with the definition of the infinite product sets from \eqref{Eq symbolic slice}.

Note that
\begin{equation*}
\pi_{m} ( \widetilde{F} _\omega ) \subseteq F_{\pi_n (\omega) },
\end{equation*}
but the two sets might not be equal if $\pi_n (\omega) \in [0,1]$ admits another base-$n$ expansion in $\widetilde{F}$. But we always have  that \[
  F_y = \bigcup_{\omega\in\pi_m^{-1}(y)} \pi_m(\widetilde{F}_\omega )
\]
This is a union of at most two sets (again, if one pre-image of $y$ is not in $\widetilde{F}$, the corresponding term in the union is empty). Given $\omega$, we have
\begin{equation*}
\pi_m(\widetilde{F}_\omega) = \lbrace \sum_{k=1} ^\infty \frac{x_k}{m^k}:\quad x_k \in \Gamma_{\omega_k} \rbrace.
\end{equation*}

We also have an elementary expression for the Hausdorff dimension of projections of symbolic slices: given $\omega \in [n]^\mathbb{N}$, by Billingsley's Lemma,
\begin{equation}
\label{eq:slice-dimension}
   \dim_H \pi_m(\widetilde{F}_\omega) = \liminf_{k\to\infty} \frac{\sum_{i=1}^k \log|\Gamma_{\omega_i}|}{k\log m} \leq \max_{i\in [n]} \frac{\log |\Gamma_i|}{\log m}.
\end{equation}
If, in addition, $\omega$ is generic with respect to some ergodic measure $\alpha \in P([n]^\mathbb{N})$, then by the ergodic Theorem
\begin{equation*}
\dim_H \pi_m(\widetilde{F}_\omega) = \sum_{i=0}^{n-1} \alpha([i]) \frac{ \log|\Gamma_{\omega_i}|}{\log m} = \dim_B \pi_m(\widetilde{F}_\omega).
\end{equation*} 

\subsubsection{Microsets of Bedford-McMullen carpets} \label{Section micro of BMC}
In \cite{bandt2013local}, Bandt and K{\"a}enm{\"a}ki  had studied the structure of microsets of a general class of self affine carpets, where the point of magnification is drawn according to a self affine measure. Now, Let $F$ be a Bedford-McMullen carpet, and suppose that $F$ is not a self similar set. In our recent work with Hochman \cite{algom2016self}, we were able to characterize  $m$-adic microsets of $F$ about any point in $F$. As this characterization is key for our present work, we briefly recall it.

 For $\omega\in [n]^\mathbb{N}$  and $s\in [0,1)$ we define an $(\omega,s)$-set to be a set of the form
\begin{equation} \label{Location}
\left(\begin{array}{cc}
1 & 0\\
0 & n^s
\end{array}\right)\cdot\left(\pi_{m}(\widetilde{F}_{\omega})\times P_{2}(F)+z\right)
\end{equation}
which is contained in $[-2,2]^{2}$. For a fixed $(\omega,s)$, a set $Y \subseteq [-1,1]^2$ is an  $(\omega,s)$-multiset if there are finitely many
$(\omega,s)$-sets $Y_{1},\ldots,Y_{N}$ and $z\in \pi_{m}(\widetilde{F}_{\omega})\times P_2 (F)$,
such that 
\begin{equation}\label{eq:2}
\left( \left( \begin{array}{cc}
1 & 0\\
0 & n^s
\end{array}\right)\cdot\left(\pi_{m}(\widetilde{F}_{\omega})\times P_{2}(F)-z \right)  \right) \bigcap(-1,1)^{2} \quad \subseteq Y  \subseteq \quad \bigcup_{i=1}^{N}Y_{i}\cap[-1,1]^{2}.
\end{equation}
Finally, for $\omega \in [n]^\mathbb{N}$ let $S(\omega)\subseteq [n]^\mathbb{N }\times \mathbb{T}$
denote the set \begin{equation}
  \label{eq:Su}
  S(\omega)=\{(\xi,s)\in [n]^\mathbb{N} \times \mathbb{T}\,:\,(\sigma^{l_k}\omega,l_k \log_nm)\to (\xi,s) \textrm{ for some }l_k\to\infty\}
\end{equation}
i.e. $S(\omega)$ is the set of accumulation points of the orbit of $(\omega,0)$
under the transformation $(\xi,s)\mapsto(\sigma \xi,s+\log_nm)$. 

For $\omega \in [n]^\mathbb{N}$, let $\overline{\omega}=\omega$ if $\pi_n(\omega)$ has a unique base-$n$ expansion, and  otherwise let $\overline{\omega}$ be the other expansion. Recall the definition of $m$-adic microsets from \eqref{Def M-adic microset}.

\begin{theorem} (\cite{algom2016self},Theorem 4.2) \label{Theorem - structure of tangnet sets}
Fix $f=(x,y)\in F$ with $y\neq 0,1$ and let $\omega\in\pi_n^{-1}(y)$. Then for every $m$-adic microset $T$ about $f$, there exists $(\xi,s)\in S(\omega)$ such that $T$ is a non-empty union of a $(\xi,s)$-multiset and a $(\overline{\xi},s)$-multiset. Conversely, if $(\xi,s)\in S(\omega)$, then there is an $m$-adic microset set $T$ about $f$  which is a union of this type.

In the special case when $y=0$ or $y=1$, the same is true but omitting the $(\overline{\xi},s)$-multiset from the union.
\end{theorem}

In applications, we shall either not care about the identity of the limit point $(\xi,s$), provided in the theorem, or else we will control it by starting with  $y$  whose expansions are suitably engineered. 

For general microsets, we have the following result:
\begin{theorem} \label{Theorem - structure of tangent sets general}
Let $M_k \subseteq [-1,1]^2$ be a sequence of mini-sets of $F$ of the form
\begin{equation*}
\alpha_k (F-z_k)\cap [-1,1]^2, \quad z_k \in \mathbb{R}^2, \alpha_k \rightarrow \infty.
\end{equation*} 
Then for every limit $M$ of $M_k$ in the Hausdorff metric there is some $p\in \mathbb{N}$ such that
\begin{equation*}
M \subseteq \bigcup_{i=1} ^p \left( \begin{pmatrix}
a_i & 0 \\
0 & b_i \\
\end{pmatrix} \cdot Y_i + t_i \right)
\end{equation*}
where for every $i$, $Y_i$ is an $(\omega_i,s_i)$ set for some $(\omega_i,s_i)\in [n]^\mathbb{N} \times \mathbb{T}$, $a_i,b_i >0$ and $t_i\in \mathbb{R}^2$.
\end{theorem}

The Theorem follows by inspecting the proof of Theorem 4.2 in \cite{algom2016self}, which deals with the case when there is some $z\in \mathbb{R}^2$ such that all the $z_k$'s from Theorem \ref{Theorem - structure of tangent sets general} are equal to $z$. Indeed, one notes that the results of Section 7.2, most notably a rescaled version of Corollary 7.5, generalize to this situation, with some minor modifications.  

\subsubsection{CP distributions generated by self affine measures on Bedford-McMullen carpets}
Let $F$ be a Bedford-McMullen carpet with exponents $(m,n)$, and allowed digits set $\Gamma \subset [m]\times [n]$. Recall that $\mu \in P([0,1]^2)$ is a self affine measure on $F$ if there exists some Bernoulli measure $\nu \in P( \Gamma^\mathbb{N})$ such that
\begin{equation*}
\mu = \pi_m \times \pi_n (\nu).
\end{equation*}
Notice that $P_2 \mu$ is a self similar measure on the deleted digit set $P_2 (F)$.

Given any measure $\mu \in P([0,1]^2)$ we denote by $\lbrace \mu_y \rbrace$ the family of conditional measures obtained by disintegrating $\mu$ according to the coordinate projection  $P_2 : \mathbb{R}^2 \rightarrow \mathbb{R}$, $P_2(x,y)=y$.  We also have a corresponding family of conditional measures $\lbrace \nu_\omega \rbrace$ associated with any measure $\nu\in P(\Gamma^\mathbb{N})$, obtained by disintegrating $\nu$ according to the coordinate projection $(\eta,\omega) \mapsto \omega$. 

The following Theorem, due to Fraser, Ferguson and Sahlsten, shows that self affine measures on Bedford-McMullen carpets generate ergodic CP distributions, in the sense of Definition \ref{Defition CP}.

\begin{theorem} \cite{Fraser2015Ferguson} \label{Theorem Fraser Ferguson}
Let $\mu$ be a self affine measure on a Bedford-McMullen carpet $F$ with exponents $(m,n)$. Then there is a family of boxes $\mathcal{E}$ and a $\delta$-regular partition operator $\Delta$ such that  $\mu$ generates an ergodic CP-distribution in the sense of Definition \ref{Definition CP gen}.

The measure component of the CP-distribution is the distribution of the measures of the form
\begin{equation*}
\begin{pmatrix}
m^{\frac{t}{2}} & 0 \\
0 & m^{\frac{-t}{2}} \\
\end{pmatrix} \left( \mu_y \times P_2 \mu \right),
\end{equation*} 
where $t\in [0,1)$ is distributed according to Lebesgue if $ \frac{\log n}{\log m} \notin \mathbb{Q}$, and otherwise according to some periodic measure with respect to the translation of $\mathbb{T}$ by $\theta$,  and $\mu_y$ is a conditional measure of $\mu$ with respect to the projection $P_2$, where $y$ is drawn according to $P_2\mu$.
\end{theorem}

Finally, let $\nu \in P(\Gamma^\mathbb{N})$ be a Bernoulli measure and let $\mu\in P(F)$ be the corresponding self affine measure on $F$. We have
\begin{equation} \label{Eq conditional meausres}
\mu_y = \pi_m \nu_{\pi_n ^{-1} (y)} \text { for } P_2\mu \text{ almost every }y.
\end{equation}
Thus, letting $y=\pi_n (\omega)$ that satisfy \eqref{Eq conditional meausres}, as long as $y$ does not belong to the countably many points of the form $y=\frac{j}{n^k}$ for some $j\in \mathbb{N}$ (which is of measure zero), we have
\begin{equation*}
\supp (\mu_y) = F_y = \Fomega = \lbrace \sum_{k=1} ^\infty \frac{x_k}{m^k}:\quad  x_k\in \Gamma_{\omega_k} \rbrace,
\end{equation*}

\section{Proof of the main results} \label{Section app}
\subsection{Proof of Theorem \ref{Theorem star dimension}}
Let $F$ be a Bedford-McMullen carpet with exponents $(m,n)$ such that $\frac{\log n}{\log m} \notin \mathbb{Q}$. Let $\ell$ be a non-principal line such that $F\cap \ell \neq \emptyset$. We aim to prove that 
\begin{equation*}
\dim^* F\cap \ell := \sup \lbrace \dim_H M : \text{ M is a microset of } F\cap \ell \rbrace \leq \max_{i\in [n]} \lbrace \frac{\log |\Gamma_i|}{\log m}+ \dim P_2 (F) -1, 0 \rbrace.
\end{equation*} 
This will suffice for the proof of Theorem \ref{Theorem star dimension}, since $\overline{\dim}_B F\cap \ell \leq \dim^* F\cap \ell$ (this inequality is true for any bounded set, see Lemma 2.4.4 in \cite{bishop2013fractal}).

So, let $M$ be a microset of $F\cap \ell$. Then it is not hard to see that, by definition,  there exists a microset of $\ell$ such that $M \subseteq M'$. Similarly, there exists a microset $M''$ of $F$ such that $M\subseteq M''$.  It follows that $M\subseteq M'\cap M''$.

Now, on the one hand, every microset of $\ell$ is contained within a line $\ell'$ that has the same slope as $\ell$ (so it is still a non-principal line). On the other hand,  by Theorem \ref{Theorem - structure of tangent sets general},  $M''$ is contained within a finite union of sets of the form
\begin{equation*}
M = \bigcup_{i=1} ^p  \left( \begin{pmatrix}
a_i & 0 \\
0 & b_i
\end{pmatrix} \cdot \pi_{m}(\tilde{F}_{\omega_i}) \times P_2 (F) + t_i \right),
\end{equation*} 
where  $p<\infty$, $t_i \in \mathbb{R}^2$, $(a_i,b_i) \in \mathbb{R}^2 \setminus \lbrace (0,0) \rbrace$ , $\omega_i \in [n]^\mathbb{N}$. 

Combining these observations, we see that
\begin{equation*}
M \subseteq M' \cap M'' \subseteq \bigcup_{i=1} ^p \left( \left( \begin{pmatrix}
a_i & 0 \\
0 & b_i
\end{pmatrix} \cdot \pi_{m}(\tilde{F}_{\omega_i}) \times P_2 (F) + t_i \right)\bigcap \ell' \right)
\end{equation*}
and therefore
\begin{eqnarray*}
\dim_H M &\leq& \max_{i} \dim_H \left( \begin{pmatrix}
a_i & 0 \\
0 & b_i
\end{pmatrix} \cdot \pi_{m}(\tilde{F}_{\omega_i}) \times P_2 (F) + t_i \right)\bigcap \ell' \\
& = & \max_i \dim_H \left( \pi_{m}(\tilde{F}_{\omega_i}) \times P_2 (F)  \right)\bigcap \ell''
\end{eqnarray*}
where $\ell''$ is the corresponding non-principal affine line. An application of Theorem \ref{Theorem main Theorem} shows that
\begin{equation*}
\dim_H M \leq \max_{i\in [n]} \lbrace \frac{\log |\Gamma_i|}{\log m_1}+ \dim P_2 (F) -1, 0 \rbrace.
\end{equation*}
As required.

\subsection{Proof of Theorem \ref{Theorem intersections of carpets}}
Let $F$ and $E$ be two incommensurable Bedford-McMullen carpets, with exponents $(m_1,n_1),(m_2,n_2)$ respectively. Recall that we denote by $\Gamma \subset [m_1]\times [n_1]$ and  $\Lambda \subset [m_2]\times [n_2]$ the allowed digits sets that define $F$ and $E$, respectively. Let $g:\mathbb{R}^2 \rightarrow \mathbb{R}^2$ be an invertible affine map, such that its linear part is a diagonal matrix. We prove that
\begin{equation*}
\dim^* g(F)\cap E \leq  \max_{(i,j)\in [n_1]\times [n_2]} \lbrace  \frac{\log |\Gamma_i|}{\log m_1} + \frac{\log |\Lambda_j|}{\log m_2} -1 ,0 \rbrace +\max \lbrace \dim P_2 (F) + \dim P_2 (E) -1,0 \rbrace.
\end{equation*} 
To this end, let $M$ be a microset of $g(F) \cap E$. Then, on the one hand, $M$ is contained within a microset of $g(F)$. Since $F$ and $g(F)$ are affine images of each other, by an analogue of Proposition \ref{Proposition Covariance}, every micorset of $g(F)$ is contained within an image of a microset of $F$ under an affine map with the same linear part as $g^{-1}$. Since the linear part of $g^{-1}$ is diagonal, and by Theorem \ref{Theorem - structure of tangent sets general}, we know that 
\begin{equation*}
M \subseteq \bigcup_{i=1} ^p \left( \begin{pmatrix}
a_i & 0 \\
0 & b_i
\end{pmatrix} \cdot \pi_{m_1}(\tilde{F}_{\omega_i}) \times P_2 (F) + t_i \right), 
\end{equation*}
where  $p<\infty$, $t_i \in \mathbb{R}^2$, $(a_i,b_i) \in \mathbb{R}^2 \setminus \lbrace (0,0) \rbrace$ , $\omega_i \in [n_1]^\mathbb{N}$. On the other hand, $M$ is also contained within a Microset of $E$, so
\begin{equation*}
M \subseteq \bigcup_{j=1} ^q \left( \begin{pmatrix}
c_j & 0 \\
0 & d_j
\end{pmatrix} \cdot \pi_{m_2}(\tilde{E}_{\eta_j}) \times P_2 (E) + t_j ' \right), 
\end{equation*}
where  $q<\infty$, $t_j ' \in \mathbb{R}^2$, $(c_j,d_j) \in \mathbb{R}^2 \setminus\lbrace (0,0) \rbrace$ , $\eta_j \in [n_2]^\mathbb{N}$.

It follows that $M$ is contained within a finite union of sets of the form
\begin{equation*}
M_{i,j} = \left( \begin{pmatrix}
a_i & 0 \\
0 & b_i
\end{pmatrix} \cdot \pi_{m_1}(\tilde{F}_{\omega_i}) \times P_2 (F) + t_i \right) \bigcap \left( \begin{pmatrix}
c_j & 0 \\
0 & d_j
\end{pmatrix} \cdot \pi_{m_2}(\tilde{E}_{\eta_j}) \times P_2 (E) + t_j ' \right)
\end{equation*}
for some $1 \leq i\leq p,1\leq j \leq q$. Rewriting the equation above, we have $M_{i,j} = P_1 (M_{i,j})\times P_2 (M_{i,j})$, where
\begin{equation*}
P_1 (M_{i,j}) =  ( a_i \cdot \pi_{m_1}(\tilde{F}_{\omega_i})  + P_1(t_i))\bigcap(c_j  \cdot \pi_{m_2}(\tilde{E}_{\eta_j}) +P_1(t_j ') )
\end{equation*}
and
\begin{equation*}
P_2 (M_{i,j}) = ( b_i \cdot P_2 (F) +P_2 (t_i))\bigcap (d_j \cdot P_2 (E) + P_2 (t_j '))).
\end{equation*}

Finally, $P_1(M_{i,j})$ corresponds to a non principal slice in the product set $\pi_{m_1}(\tilde{F}_{\omega_i}) \times \pi_{m_2}(\tilde{E}_{\eta_j})$. By Theorem \ref{Theorem main Theorem}, see that 
\begin{equation*}
\overline{\dim_B } \left( P_1 (M_{i,j}) \right) \leq   \max_{(i,j)\in [n_1]\times [n_2]} \lbrace \frac{\log |\Gamma_i|}{\log m_1} + \frac{\log |\Lambda_j|}{\log m_2} -1 ,0 \rbrace 
\end{equation*}
In addition, $P_2(M_{i,j})$ corresponds to a non principal slice in the product set $P_2 (F) \times P_2 (E)$, so by Theorem \ref{Theorem main Theorem} (or by the main results of \cite{wu2016proof} and \cite{shmerkin2016furstenberg})
\begin{equation*}
\overline{\dim_B} \left( P_2 (M_{i,j}) \right) \leq \max \lbrace \dim P_2 (F) + \dim P_2 (E) -1,0 \rbrace
\end{equation*}
Since $M \subseteq \bigcup M_{i,j}$ (a finite union), and $M_{i,j} = P_1 (M_{i,j})\times P_2 (M_{i,j})$  Combining the last two displayed equations completes the proof by well known properties of $\overline{\dim}_B$.

The second part of Theorem \ref{Theorem intersections of carpets}, where the linear part of $g$ is an anti-diagonal matrix, follows by a similar argument. 

\subsection{Proof of Theorem \ref{Theorem mutual singularity}}
Let $\mu \in P(F)$ and $\nu \in P(E)$ be self affine measures that satisfy the conditions of Theorem \ref{Theorem mutual singularity} part (1). Let $g:\mathbb{R}^2 \rightarrow \mathbb{R}^2$ be an affine map such that its linear part is given by a diagonal matrix. Suppose towards a contradiction that the conclusion of  Theorem \ref{Theorem mutual singularity} part (1) is false. Then there is a mutually non-null set $A$  such that $(g\mu)|_A \sim \nu|_A$.  

Now, by Theorem \ref{Theorem Fraser Ferguson}, $\mu$ generates an ergodic CP distribution $Q_1$ in the sense of Definition \ref{Definition CP gen}. Therefore, by Theorem \ref{Theorem Gavish}, $\mu$ generates a distribution $W_1\in P(P([0,1]^2))$ at $\mu$ almost every point $x$, in the sense of Definition \ref{Generating FD}. By Lemma \ref{Lemma generating w.r.t. image and abs. conti.} part (2), $g\mu$ generates the push-forward of this distribution $D(g)(x) W_1$, at $g\mu$ almost every point. By Lemma \ref{Lemma generating w.r.t. image and abs. conti.} part (1), $(g\mu)|_A \ll g\mu$, so $(g\mu)|_A$ generates the same distribution at almost every point in $A$.

By a completely analogues argument, $\nu$ generates an ergodic CP distribution $Q_2$. Therefore,  $\nu$ generates a distribution $W_2 \in P(P([0,1]^2))$ at $\nu$ almost every point. Since $\nu|_A \ll \nu$, by Lemma \ref{Lemma generating w.r.t. image and abs. conti.}, $\nu|_A$ generates the same distribution.

Thus, the assumption that $g\mu|_A \sim \nu|_A$, implies, via Lemma \ref{Lemma generating w.r.t. image and abs. conti.} part (2), that there exists some $x\in A$ such that $L W_1 = W_2$, where $L:=D(g)(x)\in GL(\mathbb{R}^2)$ is a diagonal matrix by assumption. Let us denote this common distribution by $P$. 

Therefore, By Theorem \ref{Theorem Gavish}, for $P$ almost every $\rho$ there is a $r_1>0$  such that $\rho|_{B(0,r_1)} \ll S_t \circ T_x (L\alpha)$, where $\alpha$ is a $Q_1$ typical measure, $t>0$ and $x\in \supp(\alpha)$. Similarly, for $P$ almost every $\rho$ there is a $r_2>0$  such that $\rho|_{B(0,r_2)} \ll S_u \circ T_y (\beta)$, where $\beta$ is a $Q_2$ typical measure, $u>0$  and $y\in \supp(\beta)$. Thus, for $P$ almost every $\rho$ there is a small ball such that $\rho|_{B(0,r)}$ is absolutely continuous with respect to both $S_t \circ T_x (L\alpha)$ and $S_u \circ T_y (\beta)$. Let us select such a measures $\rho$ and corresponding measures $\alpha$ and $\beta$.

Moreover, we may assume that we chose $\alpha$ and $\beta$ such that $\dim \alpha = \dim \mu$ and $\dim \beta = \dim \nu$. This is because  by Theorem \ref{Theorem Fraser Ferguson}   $\dim \alpha = \dim \mu$ for $Q_1$ almost every $\alpha$, and $\dim \beta = \dim \nu$ for $Q_2$ almost every $\beta$. Combining this with Theorem \ref{Theorem Gavish}, shows that we can work with measures satisfying this property in the previous paragraph.

Let $B$ denote the support of $\rho|_{B(0,r)}$. Then both $S_t \circ T_x (L\alpha) (B) >0$ and $S_u \circ T_y (\beta)(B)>0$. It follow that
\begin{equation} \label{Eq kappa}
\dim_H B\geq \max \lbrace \dim S_t \circ T_x (L\alpha), S_u \circ T_y (\beta) \rbrace \geq \max \lbrace \dim \alpha, \dim \beta \rbrace = \max \lbrace \dim \mu, \dim \nu \rbrace.
\end{equation}

On the other hand, $$B\subseteq e^{-t} \cdot  \left( L(\supp (\alpha))-x \right) , \quad \text{and }  B\subseteq e^{-u}(\supp (\beta)-y). $$ 
Therefore,
\begin{eqnarray*}
\dim_H B & \leq & \dim_H \left( e^{-t} \cdot  \left( L(\supp \alpha)-x \right)  \right) \bigcap \left( e^{-u}(\supp (\beta)-y)   \right) \\
& =& \dim_H \left( e^{-t+u} \cdot  \left( L(\supp \alpha)-x \right)  \right) \bigcap \left(\supp (\beta)-y  \right) \\ 
&=&  \dim_H \left( e^{-t+u} \cdot  \left( L(\supp \alpha)-t \right)  \right) \bigcap \left(\supp (\beta) \right)
\end{eqnarray*}
where $t=x+e^{-u}y$. Recalling Theorem \ref{Theorem Fraser Ferguson}, we can deduce that $\dim_H B$ is bounded above by
\begin{equation*}
\overline{\dim}_B \left( e^{-t+u} \begin{pmatrix}
 L_1 & 0  \\
 0 & L_2 
 \end{pmatrix}  \begin{pmatrix}
 m_1 ^{-s/2} & 0  \\
 0 & m_1 ^{s/2}
 \end{pmatrix} \cdot F_y \times P_2 (F) +t' \right) \cap \left(  \begin{pmatrix}
 m_2 ^{-r/2} & 0  \\
 0 & m_2 ^{r/2}
 \end{pmatrix} \cdot E_z \times P_2 (E) \right)
\end{equation*}
\begin{equation*}
= \overline{\dim}_B \left( g_1 (F_y) \times g_2 (P_2 (F)) \right) \cap \left( E_z \times P_2 (E) \right) = \overline{\dim}_B \left( g_1 (F_y)\cap E_z \right) \times \left( g_2 \circ P_2 (F) \cap  P_2 (E) \right)
\end{equation*}
for suitable non-degenerate affine maps $g_1,g_2:\mathbb{R}\rightarrow \mathbb{R}$, where $y\in P_2 (F)$ and $z\in P_2 (E)$. By well known properties of the upper box dimension, we find that
\begin{equation} \label{Eq final}
\dim_H B\leq  \overline{\dim}_B \left( g_1 (F_y) \cap E_z \right)  + \overline{\dim}_B \left( (g_2\circ P_2 (F)) \cap P_2 (E) \right) .
\end{equation}
We obtain our desired contradiction by applying Theorem \ref{Theorem main Theorem} to bound the RHS of  equation \eqref{Eq final} from above, and using \eqref{Eq kappa} to bound the LHS of \eqref{Eq final} from below.

The proof of Theorem \ref{Theorem mutual singularity} part (2) is analogues.

\subsection{Proof of Theorem \ref{Theorem embeddings}}
Let $F$ and $E$ be two incommensurable Bedford-McMullen carpets. Recall that we are assuming that there exists some $0\leq i \leq n_1-1$ such that $|\Gamma_i|\geq 2$ and that $\dim^* E <2$.  Suppose, towards a contradiction, that there exists an affine map $g: \mathbb{R}^2 \rightarrow \mathbb{R}^2$ such that $g(F) \subseteq E$. We denote by $A \in GL(\mathbb{R}^2)$ the linear part of $g$. 

Let $\omega = (i,i,...)\in [n_1]^\mathbb{N}$. Then $\Fomega$ is equal to a self similar set $K\subseteq [0,1]$, where $K$ is generated by the self similar IFS $\lbrace f_j \rbrace_{j\in \Gamma_i}$  defined by
\begin{equation*}
\forall j\in \Gamma_i, \forall x\in \mathbb{R}, \quad  f_j (x) = \frac{x+j}{m_1}.
\end{equation*}
Thus,  $K$ is a deleted digit set, so we have
\begin{equation*}
\dim K = \frac{\log |\Gamma_i|}{\log m_1}>0.
\end{equation*}

Now, let $y = \pi_{n_1} (\omega)$, fix $(x,y)\in (\Fomega,y)\subset F$ and let $g(x,y)=(w,z)\in E$.  Consider the following two sequences of $m_1$-adic minisets of $E$ and $F$ respectively,
\begin{equation*}
M_k = \left( m_1 ^k ( E - (w,z)) \right) \cap [-1,1]^2 , \quad R_k = \left( m_1 ^k (F-(x,y)) \right) \cap [-1,1]^2. 
\end{equation*}
Find a subsequence such that both $M_{n_k}$ and $R_{n_k}$ converges. By applying Proposition \ref{Proposition Covariance} and Theorem \ref{Theorem - structure of tangnet sets} along this subsequence, we find that, since $\omega$ is a fixed points for the shift on $[n_1]^\mathbb{N}$
\begin{equation} \label{Eq cov of micro}
 A \left( \left( K \times P_2 (F)  -t \right)\bigcap (-1,1)^2 \right) \subseteq  \bigcup_{i=1} ^p \left( \begin{pmatrix}
a_i & 0 \\
0 & b_i
\end{pmatrix} \cdot \pi_{m_2}(\tilde{E}_{\eta_i}) \times P_2 (E) + t_i \right), 
\end{equation}
where  $p<\infty$, $\eta_i \in [n_2]^\mathbb{N}$, $t_i \in \mathbb{R}^2$, $(a_i,b_i) \in \mathbb{R}^2 \setminus \lbrace (0,0) \rbrace$ and $t\in \pi_{m_1} (\tilde{F}_\omega) \times P_2 (F)$ (we absorb the $c$ from Theorem \ref{Theorem - structure of tangnet sets} that should appear on the LHS into the matrices on the RHS).

By \eqref{Eq star dim of carpet}, the assumption $\dim^* E<2$ implies that either $\dim P_2 (E)<1$ or $\max_{j\in [n_2]} |\Lambda_j| < m_2$. If $\max_{j\in [n_2]} |\Lambda_j| < m_2$ then  by projecting \eqref{Eq cov of micro} by $P_1$ and using the fact that $A$ is invertible, there exist $a_1,a_2 \in \mathbb{R}$ not both zero such that
\begin{equation*}
\emptyset \neq  a_1\cdot (K-t_1)\bigcap (-1,1)+a_2\cdot (P_2(F)-t_2)\bigcap (-1,1) \subseteq \bigcup_{i=1} ^p \left( a_i \cdot \pi_{m_2} (\tilde{E}_{\eta_i}) + t_i \right).
\end{equation*}
Since both $K$ and $P_2 (F)$ are self similar sets, we see that either
\begin{equation*}
\phi( K ) \subseteq \bigcup_{i=1} ^p \left( a_i \cdot \pi_{m_2} (\tilde{E}_{\eta_i}) + t_i \right), \quad \text{ or } \psi( P_2 (F) ) \subseteq \bigcup_{i=1} ^p \left( a_i \cdot \pi_{m_2} (\tilde{E}_{\eta_i}) + t_i \right),
\end{equation*}
where either $\phi$ or $\psi$ are invertible similarity maps $\mathbb{R}\rightarrow \mathbb{R}$. 
However, since  $\frac{\log m_1}{\log m_2}, \frac{\log n_1}{\log m_2} \notin \mathbb{Q}$, both options lead to a contradiction, since e.g. if the first option holds then by Theorem \ref{Theorem main Theorem} part (1), 
\begin{eqnarray*}
0<\overline{\dim}_B K &=& \overline{\dim}_B \left( \phi(K) \right) \\
&=& \overline{\dim}_B \left( \phi(K) \bigcap \left( \bigcup_{i=1} ^p \left( a_i \cdot \pi_{m_2} (\tilde{E}_{\eta_i}) + t_i \right) \right) \right)  \\
& \leq & \dim_H K+ \max_{j\in [n_2]} \frac{\log |\Lambda_j|}{\log m_2}-1 \\
& \lneq & \dim_H K \\
\end{eqnarray*}
since $\max_{j\in [n_2]} \frac{\log |\Lambda_j|}{\log m_2} <1$ for all $j\in [n_2]$.

If $\dim P_2 (E)<1$ then we follow a similar argument, projecting \eqref{Eq cov of micro} by $P_2$ this time, and using the fact that both $\frac{\log m_1}{\log n_2}, \frac{\log n_1}{\log n_2} \notin \mathbb{Q}$ and that $A$ is invertible.

\section{A CP chain on the space of slices of a family of product sets} 
\subsection{Some notations and preliminaries}
We now begin the proof of Theorem \ref{Theorem main Theorem}. Recall the notation introduced before Theorem \ref{Theorem main Theorem}. In particular, we always assume $m_1>m_2$ and $\frac{\log m_2}{\log m_1} = \theta \notin \mathbb{Q}$. Let $R_\theta : \mathbb{T} \rightarrow \mathbb{T}$ denote the irrational rotation $R_\theta (t) = t+\theta \mod 1$.  Let $X = ([m_1] \times [m_2])^{\mathbb{N}_0}$, where $\mathbb{N}_0 = \lbrace 0 \rbrace \cup \mathbb{N}$. 

For $t\in \mathbb{T}$, define a map $\sigma_t : [n_1]^{\mathbb{N}} \rightarrow [n_1]^{\mathbb{N}}$ by 
$\sigma_t (\omega)$ = 
     $\begin{cases}
      \sigma (\omega)  &\quad\text{if } t \in [1- \theta,1) \\
       \omega &\quad\text{if } t\in [0, 1-\theta) \\
     
     \end{cases}$
     
Next, let $t\in \mathbb{T}$. We define a map $\mathbb{T}\rightarrow [2]^{\mathbb{N}_0}$ by  $t\mapsto v_t$, where 
\begin{equation*}
 v_t (n) = 1 \text{ if and only if } \Rtheta^n (t) \in [1-\theta,1).
\end{equation*}
Notice that this is an injection.  On the space $[2]^{\mathbb{N}_0}$ we use the metric $d(x,y)=m_2 ^{-\min \lbrace k : x_k \neq y_k \rbrace}$. Thus, this identification induces a  metric $d_\theta$ on the image of $\mathbb{T}$ in $[2]^{\mathbb{N}_0}$ by taking
\begin{equation*}
d_\theta( v_s,v_t)= m_2 ^{-\min \lbrace k\geq 0: \quad v_t (k) \neq v_s (k)\rbrace}.
\end{equation*}
\begin{Notation}
We denote by $S$ the closure with respect to the metric $d_\theta$ on $[2]^{\mathbb{N}_0}$, of the image of $\mathbb{T}$ under the map $t\mapsto v_t$.
\end{Notation}

 Notice that not every $\tau \in S$ has some $t\in \mathbb{T}$ such that $\tau = v_{t}$. Indeed, this follows by noting that, for a sequence $\lbrace t_k \rbrace\subset \mathbb{T}$, if $v_{t_k}$ converges to $v_t$ in $d_\theta$ for some $t\in \mathbb{T}$, then $t_k$ converges to $t$ in the usual metric on $\mathbb{T}$. Thus, for the sequence $t_k = 1-\theta - \frac{1}{k}$, $v_{t_k}$ has no $d_\theta$ limits coming from elements of $\mathbb{T}$. For if it had one then it would have to be $v_{1-\theta}$. But $v_{1-\theta - \frac{1}{k}} (1) =1 \neq 0=  v_{1-\theta} (1)$ for all $k$ large enough, a contradiction.  

We  define a partition $\mathcal{C}$ of $\mathbb{T}$ in the following manner: $\mathcal{C}= \lbrace [1-\theta,1), [0,1-\theta) \rbrace$. We also denote, for every $k\in \mathbb{N}$, the partition $\mathcal{C}_k = \bigvee_{i=0} ^{k-1} \Rtheta^{-i} \mathcal{C}$. Notice that the elements of $\mathcal{C}_k$ are half closed half open intervals.

Next, for $\tau \in S$ we define 
\begin{equation} \label{Def rkt}
r_k (\tau) = | 0\leq i \leq k-1: \tau(i) =1|,
\end{equation}
and for $t\in \mathbb{T}$, we abuse notation and write $r_k (t):=r_k(v_t)$.

\begin{Claim} \label{Claim tau t}
\begin{enumerate}
\item There exists some integer $C>1$ such that for every $\tau \in S$ and every $k\in \mathbb{N}$,
\begin{equation*}
|r_k (\tau) - k\cdot \theta| \leq C
\end{equation*}

\item Let $t\in \mathbb{T}$. Assume that for every $k$, $t$ is not an endpoint of an interval in $\mathcal{C}_k$. Suppose that a sequence $t_k$ converges to $t$ in the usual metric on $\mathbb{T}$. Then $v_{t_k}$ converges to $v_t$ in $S$.
\end{enumerate}
\end{Claim}
For a proof, see Section \ref{Section Remaining proofs}.
\subsection{Symbolic setting}

Let $\tau \in S$, and set $Z(\tau) = \lbrace n\geq 0 : \tau_n =1 \rbrace$. Write the elements of $Z(\tau)$ in increasing order $x_1 (\tau) < x_2 (\tau)<...$.  

\begin{Definition}
For every $\tau \in S$ and $(\omega,\eta) \in  [n_1]^{\mathbb{N}} \times [n_2]^{\mathbb{N}}$ we define the set $X_{\tau,\omega,\eta} \subseteq X$ as

\begin{equation*}
\left( \prod_{i=0} ^{x_1 (\tau)-1} [\lbrace 1 \rbrace \times \Lambda_{\eta_{i+1}}] \right) \times [ \Gamma_{\omega_1} \times \Lambda_{\eta_{x_1(\tau)+1}}] \times \left( \prod_{i=x_1(\tau) +1} ^{x_2 (\tau) -1}  [\lbrace 1 \rbrace \times \Lambda_{\eta_{i+1}}] \right) \times [ \Gamma_{\omega_2} \times \Lambda_{\eta_{x_2(\tau)+1}}] 
\end{equation*}
\begin{equation*}
\times \left( \prod_{i=x_2(\tau) +1} ^{x_3 (\tau) -1}  [\lbrace 1 \rbrace \times \Lambda_{\eta_{i+1}}] \right) ...
\end{equation*}
if $x_1 (\tau)\neq0$. Otherwise, the zero coordinate of $X_{\tau,\omega,\eta}$ is $ \Gamma_{\omega_1} \times \Lambda_{\eta_{1}}$, and the rest of the coordinates are defined as above. We also define
\begin{itemize}
\item A metric on $X_{\tau,\omega,\eta}$ by taking $d(x,y) = m_2 ^{-\min{ \lbrace k\geq 0: \quad  x_k \neq y_k \rbrace}}$. 

\item A map $\pi_{\tau,\omega,\eta} : X_{\tau,\omega,\eta} \rightarrow \Fomega\times \Eeta$ by taking 
\begin{equation*}
\pi_{\tau,\omega,\eta} ( (a_n, b_n) ) = ( \sum_{k=1} ^\infty \frac{a_{x_k(\tau)}}{m_1 ^k}, \sum_{k=1} ^\infty \frac{b_{k-1}}{m_2 ^k} ). 
\end{equation*}
Note that this is a surjective map.
\end{itemize}
\end{Definition}

\begin{Lemma} \label{Lemma 4.1}
\begin{enumerate}
\item Suppose $\tau_k, \tau\in S$ and $\omega_k, \omega \in [n_1] ^{\mathbb{N}}$ and $\eta_k,\eta \in [n_2]^{\mathbb{N}}$ are such that $\tau_k \rightarrow \tau$ in $d_\theta$, and $(\omega_k,\eta_k) \rightarrow (\omega,\eta)$ in $[n_1]^{\mathbb{N}}\times [n_2]^{\mathbb{N}}$. Then $X_{\tau_k, \omega_k, \eta_k} \rightarrow X_{\tau,\omega,\eta}$ in the Hausdorff metric, and $\pi_{\tau_k,\omega_k,\eta_k} \rightarrow \pi_{\tau,\omega,\eta}$ uniformly.

\item $\exists C_1 >0$ such that the maps $\pi_{\tau,\omega,\eta}$ are uniformly $C_1$-Lipschitz.

\item $\exists C_2>0$ such that $\forall \tau \in S$ and $(\omega,\eta) \in [n_1]^{\mathbb{N}}\times [n_2]^{\mathbb{N}}$ and all $A\subseteq X_{\tau,\omega,\eta}$, we have 
\begin{equation*}
N(A, m_2 ^{-k}) \leq C_2 \cdot N ( \pi_{\tau,\omega,\eta} (A), m_2 ^{-k})
\end{equation*}
 for all $k\in \mathbb{N}$, where  $N(\cdot ,m_2 ^{-k})$ on the left hand side denotes the number of $k$-level cylinders $A$ intersects, and on the left hand side the number of  $m_2 ^{-k}$-adic squares needed to cover a set in $[0,1]^2$.

\item For all $\tau \in S$ and $(\omega,\eta) \in [n_1]^{\mathbb{N}}\times [n_2]^{\mathbb{N}}$, and all $A\subseteq X_{\tau,\omega,\eta}$, $\dim_H (A) = \dim_H (\pi_{\tau,\omega,\eta} (A))$.
\end{enumerate}
\end{Lemma}

\begin{proof}
Follows along the same lines of Lemma 4.1 in \cite{wu2016proof}. The key idea here is that for any cylinder $[u]$ of length $k$ in $X_{t,\omega,\eta}$, the set $\pi_{\tau,\omega,\eta} ([u])$ is contained in at most $4$ box's of side length about  $m_2 ^{-k}\times m_2 ^{-k}$ , since by Claim \ref{Claim tau t} the length on the $x$-axis is $m_1 ^{-k\cdot \theta - C} \leq m_1 ^{-r_k (\tau)} \leq m_1 ^{-k\cdot \theta +C}$.  
\end{proof}

Next, define for every $t\in \mathbb{T}$, a map $\Phi_t : [0,1]^2 \rightarrow [0,1]^2$, by

$$\Phi_t (z) = 
     \begin{cases}
      (T_{m_1} (z_1), T_{m_2} ( z_2))  &\quad\text{if } t\in [1 - \theta,1) \\
       ( z_1, T_{m_2} (z_2))  &\quad\text{if } t\in [0, 1-\theta) \\
     
     \end{cases}$$

\begin{Lemma}
If $\ell$ is a line with slope $m_1 ^t, t\in \mathbb{T}$, through $[0,1]^2$, then $\Phi_t (\ell)$ consists of a finite number of lines of slope $m_1 ^{R_\theta (t)}$ (here we think of $\mathbb{T}$ as $[0,1]$ with addition $\pmod 1$).
\end{Lemma}

Next, let $\ell_{u,z}$ denote the line through $z$ with slope $u$. Define $$\mathcal{F} \subseteq \cmpct (X) \times X \times \mathbb{T} \times S \times [n_1]^\mathbb{N} \times [n_2]^\mathbb{N}$$ by
\begin{equation} \label{Eq for F}
 \mathcal{F} = \lbrace (A,x,t,\tau,\omega,\eta): x\in A \subseteq X_{\tau,\omega,\eta}, \quad   \pi_{\tau,\omega,\eta} (A) \subseteq [ \Fomega \times \Eeta]\cap \ell_{m_1 ^t,\pi_{\tau,\omega,\eta} (x)} \rbrace.
\end{equation}

Note that for every line $\ell_{m_1 ^t,z}$ for some $t\in \mathbb{T}$ and $z\in [\Fomega \times \Eeta]\cap \ell_{m_1 ^t,z}$, for all $x\in \pi_{v_t,\omega,\eta} ^{-1} (z)$ we have
\begin{equation*}
( \pi_{v_t,\omega,\eta} ^{-1} ([ \Fomega \times \Eeta]\cap \ell_{m_1^t,\pi_{v_t,\omega,\eta} (x)}), x, t, v_t, \omega,\eta) \in \mathcal{F}.
\end{equation*}

\begin{Lemma} \label{Lemma 4.3}
\begin{enumerate}
\item If $(A,x,t, v_t, \omega,\eta) \in F$ then
\begin{equation*}
( \sigma(A\cap [x_1]), \sigma (x), \Rtheta (t), \sigma (v_t), \sigma_t (\omega),\sigma(\eta)) \in \mathcal{F}.
\end{equation*} 

\item If $(A_k,x_k,t_k, \tau_k, \omega_k,\eta_k) \rightarrow (A,x,t, \tau, \omega,\eta)$ and $(A_k,x_k,t_k, \tau_k, \omega_k,\eta_k)  \in F$ for all $k$, then $(A,x,t, \tau, \omega,\eta)\in F$. 
\end{enumerate}
\end{Lemma}

\begin{proof}
Let $x \in X_{v_t,\omega,\eta}$. Then $\sigma(x)\in  X_{v_{\Rtheta (t)},\sigma_t (\omega), \sigma(\eta)}=X_{\sigma(v_t),\sigma_t(\omega),\sigma(\eta)}$. In particular,
\begin{equation*}
 \Phi_t ( \pi_{ v_t,\omega,\eta} (x)) = \pi_{\sigma (v_t), \sigma_t (\omega), \sigma(\eta)} (\sigma (x)).
\end{equation*}
It follows that 
\begin{equation*}
\Phi_t ( \pi_{v_t,\omega,\eta} (A\cap [x_1]))=\pi_{\sigma (v_t), \sigma_t (\omega),\sigma(\eta)} (\sigma(A\cap [x_1])).
\end{equation*}
Also, since $\tilde{F}_\omega$ and $\tilde{E}_\eta$ are product sets,
\begin{equation*}
\Phi_t ( \Fomega \times \Eeta) = \pi_{m_1} (\tilde{F}_{\sigma_t (\omega)})\times \pi_{m_2} (\tilde{E}_{\sigma (\eta)}).
\end{equation*}
This implies part (1). Part (2) follows from part (1) of Lemma \ref{Lemma 4.1}.
\end{proof}

\subsection{Construction of a CP-distribution} \label{Section 4.2}
Consider the space $Y = P(X) \times X \times \mathbb{T} \times S \times [n_1]^{\mathbb{N}} \times [n_2]^ {\mathbb{N}}$. Define $\hat{M}:Y \rightarrow Y$ by 
\begin{equation*}
\hat{M} (\mu,x,t, \tau, \omega,\eta) = (\mu^{[x_0]},\sigma (x), \Rtheta (t), \sigma(\tau), \sigma_t (\omega),\sigma(\eta)),
\end{equation*}
(recall the definition of $\mu^{[x_0]}$ from Section \ref{Section CP symb}).  This rather cumbersome space comes from mostly natural geometric considerations: the first two coordinates of $Y$ are the usual (symbolic) setting for a CP distribution, as in Definition \ref{Def CP dist}. These will describe measures on slices of $\Fomega\times \Eeta$. The following two coordinates, $\mathbb{T} \times S$, capture the slope of the slice, where the $S$ coordinate (the only "unnatural" coordinate) is needed for compactness reasons. The final two coordinates capture the $(\omega,\eta)$ that corresponds to the set $\Fomega\times \Eeta$ in the family that is being sliced.

Note that $\hat{M}$ is not continuous; the set of its discontinuity points is contained in $P(X)\times X \times \lbrace 0, 1-\theta \rbrace \times S \times [n_1]^\mathbb{N} \times [n_2] ^\mathbb{N}$. This is because the skew-product $\sigma : \mathbb{T} \times [n_1] ^\mathbb{N} \rightarrow [n_1] ^\mathbb{N}$ is continuous at all points except at $\lbrace 0 , 1-\theta\rbrace \times [n_1] ^\mathbb{N}$.

We shall say that $P\in P(Y)$ is globally adapted if for every $f\in C(Y)$,
\begin{equation} \label{Glob. ada.}
\int f(\mu,x,t, \tau, \omega,\eta) dP(\mu,x,t, \tau, \omega,\eta) = \int \left( \int f(\mu,x,t, \tau, \omega,\eta)d \mu (x)\right) dP_{1,3,4,5,6} (\mu,t, \tau, \omega,\eta)
\end{equation}
where $P_{1,3,4,5,6}$ denote the corresponding marginal of $P$ on the coordinates $1,3,4,5,6$.   This means that if a property holds $P$ a.e. then it holds for $P_{1,3,4,5,6}$ a.e. $(\mu,t,\tau,\omega,\eta)$ and for $\mu$ a.e. $x$.

For $P\in P(Y)$ define $H(P) = \int \frac{1}{\log m_2} \log \mu ([x_1]) dP_{1,2} (\mu,x)$.
\begin{Proposition} \label{Lemma CP chain}
Suppose that $\exists (t_0, \omega_0,\eta_0)\in \mathbb{T} \times [n_1]^\mathbb{N} \times [n_2] ^\mathbb{N}$ such that there exists a line $\ell$ of slope $m_1 ^{t_0}$ that satisfies 
\begin{equation*}
\overline{\dim}_B \left( \pi_{m_1} (\tilde{F}_{\omega_0})\times \pi_{m_2} (\tilde{E}_{\eta_0}) \right) \bigcap \ell = \gamma >0.
\end{equation*}

Then there exists  $Q\in P(Y)$ that is $\hat{M}$ invariant, $H(Q)=\gamma$, and $Q$ satisfies \eqref{Glob. ada.} for all $f\in C(Y)$. In particular, $Q_{1,2}$ is a CP-distribution. Moreover:
\begin{enumerate}
\item Recall the definition of $\mathcal{F}$ from \eqref{Eq for F}, and let 
\begin{equation*}
D_\mathcal{F} = \bigcup_{(A,x,t, \tau, \omega,\eta) \in \mathcal{F}} P(A)\times \lbrace x \rbrace \times \lbrace t \rbrace \times \lbrace \tau \rbrace \times \lbrace \omega \rbrace \times \lbrace \eta \rbrace.
\end{equation*}
Then $Q$ is supported on $D_\mathcal{F}$. Thus, $Q$ a.e. ergodic component is supported there. Moreover,
\begin{equation} \label{Eq correct coding}
Q ( \lbrace (\mu,x,t,\tau, \omega,\eta) \in D_\mathcal{F}  : \tau = v_t \rbrace)=1.
\end{equation}

\item There a measurable set $E_\gamma \subset Y$ such that $Q(E_\gamma )>0$,  and for $Q$ a.e. $(\mu,x,t, \tau, \omega,\eta)\in E_\gamma$, $Q_{1,2} ^{(\mu,x,t, \tau, \omega,\eta)}$ (the marginal of the corresponding ergodic component of $Q$ on the first two coordinates) is an ergodic CP chain of dimension $\geq \gamma$. 

\end{enumerate}
\end{Proposition} 

\begin{proof}
Let $E = \pi_{v_{t_0},\omega_0,\eta_0} ^{-1} \left( \left( \pi_{m_1} (\tilde{F}_{\omega_0})\times \pi_{m_2} (\tilde{E}_{\eta_0}) \right) \bigcap \ell \right)$. By Lemma \ref{Lemma 4.1} parts (2) and (3), we have $\overline{\dim}_B (E)=\gamma$ in the space $X_{v_{t_0},\omega_0,\eta_0}$. Thus, there exists $n_k \nearrow \infty$ such that 
\begin{equation}
\lim_k \frac{ \log N(E, m_2 ^{-n_k})}{-n_k \log m_2} = \gamma.
\end{equation}

Define a sequence of measures $\lbrace \mu_k \rbrace_k$ on $E$ by setting 
\begin{equation*}
\mu_k = \frac{1}{N(E,m_2^{-n_k})} \sum_{u: |u|=n_k, [u]\cap E \neq \emptyset} \delta_{x_u}
\end{equation*}
for some $x_u \in [u]\cap E$. We also define
\begin{equation*}
P_k = \frac{1}{N(E,m_2 ^{-n_k})} \sum_{u: |u|=n_k, [u]\cap E \neq \emptyset} \delta_{(\mu_k, x_u, t_0, v_{t_0}, \omega_0,\eta_0)}
\end{equation*}
\begin{equation*}
Q_k = \frac{1}{n_k} \sum_{i=0} ^{n_k-1} \hat{M}^i P_k.
\end{equation*}
Note that by the construction of $Q_k$, for all $f\in C(Y)$, \eqref{Glob. ada.} holds.

Note that $H(Q_k) \rightarrow \gamma$ as $k$ grows to $\infty$. The proof of this fact can be found in section 4.2 of \cite{wu2016proof}. Now, by the compactness of $P(Y)$, we may find a sub sequence such that $Q_k \rightarrow Q \in P(Y)$. It follows that $H(Q)=\gamma$.

We claim that $Q$ is $\hat{M}$ invariant. Indeed, we note that $Q_3$ is a measure that is invariant under the irrational rotation $\Rtheta$. Therefore, this must be the Lebesgue measure on $\mathbb{T}$. Thus,
\begin{equation*}
Q ( \text{ discontinuouties of } \hat{M} ) \subseteq Q (P(X)\times X \times \lbrace 0, 1-\theta \rbrace \times S \times [n_1]^\mathbb{N} \times [n_2]^\mathbb{N}) = Q_3 ( \lbrace 0, 1-\theta \rbrace) = 0.
\end{equation*}
It follows that $Q$ is a measure such that an orbit of $\hat{M}$ equidistributes for, and by the above calculation $\hat{M}$ is continuous up to a $Q$ null set. Therefore, $Q$ is $\hat{M}$ invariant. Finally, $Q$ satisfies \eqref{Glob. ada.} since each $Q_k$ does.

Let $Q = \int Q^{(\mu,x,t, \tau, \omega,\eta)} dQ(\mu,x,t, \tau, \omega,\eta)$ denote the ergodic decomposition of $Q$. Define
\begin{equation*}
E_\gamma = \lbrace (\mu,x,t, \tau, \omega,\eta)\in Y: \quad  H( Q^{(\mu,x,t, \tau, \omega,\eta)} \geq \gamma \rbrace.
\end{equation*}
Since $H(Q)=\gamma$ we have $Q( E_\gamma )>0$, and for $Q$ a.e. $(\mu,x,t, \tau, \omega,\eta)\in E_\gamma$, $Q_{1,2} ^{(\mu,x,t, \tau, \omega,\eta)}$ is an ergodic CP distribution of dimension $\geq \gamma$ (by Proposition \ref{Proposition components of CP}).  

Next,  by Lemma\footnote{In fact, this requires some work. Specifically, use the fact that if $\mu_k \rightarrow \mu$ in a compact metric space $X$, then $\supp(\mu) \subseteq \lbrace x: \limsup d_X(z, \supp(\mu_k))=0\rbrace$, and Lemma \ref{Lemma 4.1} part (1). } \ref{Lemma 4.3}, $D_\mathcal{F}$ is closed and $Q_k$ is supported on $D_\mathcal{F}$ for all $k$. It follows that $Q$ is supported on $D_\mathcal{F}$. Thus, $Q$ a.e. ergodic component is supported there.

Finally, we prove equation \eqref{Eq correct coding}. Let $A$ denote the countable set of all endpoints of the intervals in the partitions $\mathcal{C}_k$ of $\mathbb{T}$ for all $k\geq 0$, defined before Claim \ref{Claim tau t}. Note that this a countable set. Now, consider the set
\begin{equation*}
B = \lbrace (\mu,x,t,\tau, \omega,\eta) \in D_\mathcal{F} \cap \supp (Q):\quad  t\notin A \rbrace.
\end{equation*} 
Since the projection from the space $Y$ to its third coordinate is continuous, and since $A\subset \mathbb{T}$ is countable (and hence measurable), it follows that $B$ is a measurable set. We now prove that if $(\mu,x,t,\tau,\omega,\eta)\in B$ then $\tau = v_t$. 

Fix $(\mu,x,t,\tau,\omega,\eta)\in B$.  It is well known that since $Q$ is the weak limit of the distributions $Q_k$ (defined earlier), then
\begin{equation*}
\supp (Q) \subseteq \lbrace z\in Y: \quad \limsup_k d_Y (z, \supp (Q_k))=0 \rbrace.
\end{equation*}
Thus, there exists a sequence $n_k$ and elements $(\mu_{n_k} ,x_{n_k},t_{n_k},\tau_{n_k}, \omega_{n_k},\eta_{n_k}) \in \supp (Q_{n_k})$ that converge to $(\mu,x,t,\tau,\omega,\eta)$. In particular, $t_{n_k}$ converges to $t$ in $\mathbb{T}$ and $\tau_{n_k}$ converges to $\tau$ in $S$. Since these are elements in the support of $Q_k$, it follows that
\begin{equation*}
t_{n_k} = \Rtheta^{n_k} (t_0) =  n_k \cdot \theta + t_0 \mod 1
\end{equation*}  
\begin{equation*}
\tau_{n_k} = \sigma^{n_k} (v_{t_0}) = v_{n_k\cdot \theta +t_0 \mod 1} = v_{t_{n_k}}.
\end{equation*}
Now, since $t_{n_k}$ converges to $t$ in $\mathbb{T}$, and since $t \notin A$, we may apply part (2) of Claim \ref{Claim tau t} and see that $v_{t_{n_k}} = \tau_{n_k}$ converges to $v_t$. Therefore, $\tau = v_t$. 

Finally,
\begin{equation*}
Q(B^C)=Q( \lbrace (\mu,x,t,\tau, \omega,\eta) \in D_\mathcal{F}:\quad  t\in A \rbrace) = Q_3 (A) = 0,
\end{equation*}
since the marginal of $Q$ on the third coordinate is the Lebesgue measure and $A$ is countable. Also, notice that
\begin{equation*}
\lbrace (\mu,x,t,\tau, \omega,\eta) \in D_\mathcal{F} :\quad  \tau = v_t \rbrace = B \cup \lbrace (\mu,x,t,\tau, \omega,\eta) \in D_\mathcal{F} :\quad  \tau = v_t, t\in A \rbrace,
\end{equation*}
and since both sets on the right hand side are measurable, so it the set on the left hand side.
\end{proof}

Notice that Proposition \ref{Lemma CP chain} assumes nothing about $(\omega_0,\eta_0)$, and thus forms the first step towards the proof of Theorem \ref{Theorem main Theorem} part (1). We next discuss some improvements of the above Proposition when we can control $\omega_0$ and $\eta_0$ in the statement.  Consider the compact metric space $\mathbb{T} \times [n_1]^{\mathbb{N}} \times [n_2]^\mathbb{N}$. Define a map $$Z:\mathbb{T} \times [n_1]^\mathbb{N}\times [n_2]^\mathbb{N} \rightarrow \mathbb{T} \times [n_1]^\mathbb{N}\times [n_2]^\mathbb{N}$$ by taking 
\begin{equation*}
Z(t,\omega,\eta)=(\Rtheta (t),  \sigma_t (\omega),\sigma(\eta)).
\end{equation*}
The following Proposition was  proved  during  the proof of Proposition 4.3 in \cite{Fraser2015Ferguson}. 
\begin{theorem} \label{Theorem X ergodic MPS} \cite{Fraser2015Ferguson}
Let $\alpha_1, \alpha_2$ be Bernoulli measures on $[n_1]^\mathbb{N}$, $[n_2]^\mathbb{N}$, respectively.   Then for every $t\in \mathbb{T}$ there is a set of full $\alpha_1\times \alpha_2$ measure $A$, such that for all $(\omega,\eta)\in A$, we have
\begin{equation*}
\frac{1}{N} \sum_{i=0} ^{N-1} \delta_{Z^i (t,\omega,\eta)} \rightarrow \lambda\times \rho \times \alpha.
\end{equation*}
In particular, the measure preserving system $(\mathbb{T} \times [n_1]^\mathbb{N}\times [n_2]^\mathbb{N},Z,\lambda\times \alpha_1\times \alpha_2)$ is ergodic, where $\lambda$ is the Lebesgue measure on $\mathbb{T}$.
\end{theorem}

The following Corollary is thus a result of the previous Corollary, and the construction carried out in Proposition \ref{Lemma CP chain}:

\begin{Corollary} \label{Cor joint dist}
Let $Q$ be the CP-distribution built in Proposition \ref{Lemma CP chain}. Suppose that the $(\omega_0,\eta_0) \in [n_1]^\mathbb{N} \times [n_2]^\mathbb{N}$  appearing in the statement of Lemma \ref{Lemma CP chain} are typical with respect to $t_0$ and some Bernoulli measures $\alpha_1 \in P( [n_1] ^\mathbb{N})$ and $\alpha_2 \in P([n_2]^\mathbb{N})$, in the sense of Theorem \ref{Theorem X ergodic MPS}.  Let $Q_{3,5,6}$ denote the joint distribution of $Q$ on the third coordinate, the fifth coordinate and the sixth coordinate. Then 
\begin{equation*}
Q_{3,5,6} = \lambda \times \alpha_1 \times \alpha_2 \in P( \mathbb{T} \times [n_1]^\mathbb{N} \times [n_2]^\mathbb{N})
\end{equation*}
\end{Corollary}

\section{A skew product dynamical system}
\subsection{The transformation U} \label{Section 5.1}
Define $U: [0,1]^2 \times \mathbb{T} \times [n_1]^\mathbb{N} \times [n_2] ^\mathbb{N} \rightarrow [0,1]^2 \times \mathbb{T} \times [n_1]^\mathbb{N}  \times [n_2] ^\mathbb{N}$ by setting 
\begin{equation*}
U(z,t,\omega,\eta) = ( \Phi_t (z), \Rtheta (t), \sigma_t (\omega),\sigma(\eta)).
\end{equation*}
We denote by $U_t ^k (z)$ the first coordinate of the map $U^k (z,t,\omega,\eta)$. Recall that
\begin{equation*}
r_k (t) = | \lbrace 0 \leq i \leq k-1 : R_\theta ^i (t) \in [1-\theta,1) \rbrace|.
\end{equation*} 
Note that for $k\geq 1$, $$\sigma_t ^k (\omega):=\sigma_{\Rtheta^{k-1} (t)} \circ ... \circ\sigma_{\Rtheta (t)} \circ \sigma_{t}(\omega)  = \sigma^{r_k (t)} (\omega).$$ Thus, recalling the definition of the maps $T_{m_i}$ from \eqref{Eq T_m}
\begin{equation} \label{Eq Utk}
U^k _t (z) = \Phi_{\Rtheta^{k-1} (t)} \circ ... \circ\Phi_{\Rtheta (t)} \circ \Phi_{t}(z) =(T_{m_1} ^{r_k (t)}  (z_1) , T_{m_2} ^k  (z_2)), \text{ for } z=(z_1,z_2).
\end{equation}

We now define a sequence of partitions of $[0,1]^2 \times \mathbb{T} \times [n_1]^\mathbb{N} \times [n_2] ^\mathbb{N}$ as follows: Let $\mathcal{D}_{m_1}$ and $\mathcal{D}_{m_2}$ be the $m_1$-adic and $m_2$-adic partitions, respectively, of $[0,1]$. Recall that $\mathcal{C} = \lbrace [0,1-\theta), [1-\theta,1)\rbrace$ is the partition of $\mathbb{T}$ we previously defined. Let $\mathcal{I}_1 ^1$ be the first generation cylinder partition of $[n_1] ^\mathbb{N}$, and let $\mathcal{I}_2 ^1$ denote the first generation cylinder partition of $[n_2]^\mathbb{N}$. Similarly, Let $\mathcal{I}_1 ^k, \mathcal{I}_2 ^k$ be the first $k$ coordinate cylinder partitions of $[n_1]^\mathbb{N}, [n_2]^\mathbb{N}$, respectively.   Let 
\begin{equation} \label{Equation 5.1}
\mathcal{B}_1 = [ \mathcal{D}_m \times \mathcal{D}_n] \times \mathcal{C} \times \mathcal{I}_1 ^1 \times \mathcal{I}_2 ^1 .
\end{equation}
and for $k\geq 2$ let 
\begin{equation*}
\mathcal{B}_k = \bigvee_{m=0} ^{k-1} U^{-m} \mathcal{B}_1 .
\end{equation*}

Let us make the following observations. For $k\geq 2$, let $\mathcal{C}_k = \bigvee_{m=0} ^{k-1} \Rtheta^{-m} \mathcal{C} $, and notice that $\bigvee_{m=0} ^{k-1} \sigma^{-m} \mathcal{I}_2 ^1= \mathcal{I}_2 ^k$. 
\begin{itemize}
\item For $k\geq 1$ and $t\in \mathbb{T}$ define $\mathcal{I}_k ^t : = \bigvee_{m=0} ^{k-1} \sigma_t ^{-m} \mathcal{I}_1 ^1$. Then $\mathcal{I}_k ^t = \mathcal{I}_1 ^{r_k (t)}$.

\item For $k\geq 1$ and $t\in \mathbb{T}$ define $\mathcal{A}_k ^t = \bigvee_{m=0} ^{k-1} \Phi_t ^{-m} (\mathcal{D}_{m_1} \times \mathcal{D}_{m_2})$. By equation \eqref{Eq Utk} we have 
\begin{equation} \label{equation for A_kt}
\mathcal{A}_k ^t =   \mathcal{D}_{m_1 ^{r_k (t)}} \times \mathcal{D}_{m_2 ^k}.
\end{equation}

\item Note that if $t,t'\in \mathbb{T}$ belong to the same atom of $\mathcal{C}_k$ then $\mathcal{A}_k ^t = \mathcal{A}_k ^{t'}$ and $\mathcal{I}_k ^t = \mathcal{I}_k ^{t'}$, since this means that $r_i (t) = r_i (t')$ for all $0 \leq i \leq k-1$. 

\item Every atom of $\mathcal{B}_k$ has the form $A\times C \times I\times J$ for $J\in \mathcal{I}_2 ^k$, $C\in \mathcal{C}_k$ and $A\in \mathcal{A}_k ^t$ and $I\in \mathcal{I}_k ^t$, for some $t\in C$.
\end{itemize}
The following Lemma is modelled after Lemma 5.1 \cite{wu2016proof}. We defer its proof to section \ref{Section Remaining proofs}.

\begin{Lemma} \label{Lemma 5.1}
\begin{enumerate}
\item Let $(t,v_t,\omega,\eta)\in \mathbb{T} \times S \times [n_1]^\mathbb{N} \times [n_2] ^\mathbb{N}$ and $x\in X_{v_t,\omega,\eta}$. If $\pi_{v_t,\omega,\eta} (x)$ is in the interior  $\mathcal{A}_k ^t (\pi_{v_t,\omega,\eta} (x))$  then the set  $\pi_{v_t,\omega,\eta} ([x_0 ^{k-1}])$ is contained within $\mathcal{A}_k ^t (\pi_{v_t,\omega,\eta} (x))$ except possibly at the boundary points of  $\mathcal{A}_k ^t (\pi_{v_t,\omega,\eta} (x))$

\item Let $(\mu,x,t, v_t, \omega,\eta) \in D_\mathcal{F}$. If $\mu$ is not atomic then for $\mu$ a.e. $x$ and all $k\geq 1$,
\begin{equation} \label{Equation 5.3}
\pi_{\sigma^k (v_t), \sigma_t ^k (\omega),\sigma^k (\eta)} \left( \frac{\sigma^k( \mu|_{[x_0 ^{k-1}]})}{\mu ([x_0 ^{k-1}])} \right) = U_t ^k \left( \frac{\pi_{v_t,\omega,\eta} \mu |_{\mathcal{A}_k ^t (\pi_{v_t,\omega,\eta} (x))} }{\pi_{v_t,\omega,\eta} \mu  (\mathcal{A}_k ^t (\pi_{v_t,\omega,\eta} (x)))} \right)
\end{equation}
\end{enumerate}
\end{Lemma}

Let $\nu \in P([0,1]^2)$ and $z\in \supp (\nu)$. Denote
\begin{equation} \label{Equation 5.4}
\nu^{\mathcal{A}_k ^t (z)} = U_t ^k \left( \frac{\nu|_{\mathcal{A}_k ^t (z)}}{\nu (\mathcal{A}_k ^t (z))} \right).
\end{equation}
Note that if $\nu \in P( \ell \cap [\Fomega\times \Eeta])$ with $\ell$ being a line with slope $m_1 ^t$ then 
\begin{equation*}
\nu^{\mathcal{A}_k ^t (z)} \in P(\ell ' \cap [\pi_{m_1} ( \tilde{F}_{\sigma^k _t (\omega)})\times \tilde{E}_{\sigma^k (\eta)} ]),
\end{equation*}
where $\ell'$ has slope $m_1 ^{\Rtheta ^k (t)}$.

\subsection{Construction of U invariant measures}
In this section we construct a family of $U$ invariant measures on $[0,1]^2 \times \mathbb{T} \times [n_1]^\mathbb{N} \times [n_2] ^\mathbb{N}$ by transferring information from an ergodic component of the CP-distribution $Q$ constructed in Proposition \ref{Lemma CP chain}, in a similar spirit to (\cite{wu2016proof}, Proposition 5.3). Unlike the proof in \cite{wu2016proof}, we do this by considering the intensity measure of some of the ergodic components of the CP-chain $Q$:
\begin{theorem}
For $Q_{1,3,4,5,6}$ a.e. $(\mu,t,\tau,\omega,\eta)$ and $\mu$ a.e. $x$ s.t. $(\mu,x,t,\tau, \omega,\eta) \in E_\gamma$, let
\begin{equation} \label{Equation 5.8}
\nu^{(\mu,x,t,\tau, \omega,\eta)} :=\int \left( \pi_{\tau',\xi ,\zeta} (\nu)\times \delta_s  \times \delta_\xi \times \delta_\zeta \right) dQ^{(\mu,x,t,\tau, \omega,\eta)} _{1,3,4,5,6} (\nu,s, \tau', \xi,\zeta)ץ
\end{equation}
Then this is measure on $[0,1]^2 \times \mathbb{T} \times [n_1]^\mathbb{N} \times [n_2]^\mathbb{N}$ that is $U$ invariant.
\end{theorem}

\begin{proof}
First, notice that by Proposition \ref{Lemma CP chain} part (3) we have that for $Q$ a.e. $(\mu,x,t,\tau, \omega,\eta)$,
\begin{equation*}
\nu^{(\mu,x,t,\tau, \omega,\eta)} =  \int \left( \pi_{v_s,\xi,\zeta} (\nu)\times \delta_s  \times \delta_\xi \times \delta_\zeta \right) dQ^{(\mu,x,t,\tau, \omega,\eta)} _{1,3,4,5,6} (\nu,s, v_s, \xi,\zeta)  
\end{equation*}
(that is, $\tau'=v_s$ in \eqref{Equation 5.8}) since the set where $\tau' \neq v_s$ has $Q$ measure $0$. So, fix such an element in $E_\gamma$. Now, let $f\in C([0,1]^2 \times \mathbb{T} \times [n_1]^\mathbb{N} \times [n_2] ^\mathbb{N})$. Then, since the function we are integrating against does not depend on the second coordinate $y$ (in the space $Y$),  and by $\hat{M}$ invariance,

\begin{equation*} 
\begin{split}
\int f d\nu^{(\mu,x,t,\tau, \omega,\eta)} & = \int \left( \int f(z,s,\xi,\zeta) d\pi_{v_s,\xi,\zeta} (\nu)(z) \right) dQ^{(\mu,x,t,\tau, \omega,\eta)} _{1,3,4,5,6}  (\nu,s, v_s, \xi,\zeta) \\
& = \int \left( \int f(z,s,\xi,\zeta) d\pi_{v_s,\xi,\zeta} (\nu)(z) \right) dQ^{(\mu,x,t,\tau, \omega,\eta)} (\nu,y,s, v_s, \xi,\zeta) \\
 & = \int \left( \int f(z,s,\xi,\zeta) d\pi_{v_s,\xi,\zeta} (\nu)(z) \right) d\hat{M}Q^{(\mu,x,t,\tau, \omega,\eta)}   (\nu,y,s, v_s, \xi,\zeta) \\
\end{split}
\end{equation*}
\begin{equation*}
\begin{split}
& = \int \left( \int f(z,\Rtheta(s),\sigma_s(\xi),\sigma(\zeta)) d\pi_{\sigma(v_s),\sigma_s(\xi),\sigma(\zeta)} (\nu^{[y_0]})(z) \right) dQ^{(\mu,x,t,\tau, \omega,\eta)}  (\nu,y,s, v_s, \xi,\zeta)
\end{split}
\end{equation*}
as $\sigma(v_s)=v_{\Rtheta(s)}$, using the adaptedness\footnote{Notice that the function $(\nu,s, v_s, \xi,\zeta) \mapsto \pi_{v_s,\xi,\zeta} \times \delta_{s} \times \delta_\xi \times \delta_\zeta$ is $Q$ a.e. continuous.} of $Q$, and noting that there are finitely many options ($m_1\cdot m_2$ to be precise) for $y_0$ (so that the integrand above a simple function with respect to $y$), we have

\begin{equation*}
 = \int \left( \int \left( \int f(z,\Rtheta(s),\sigma_s(\xi),\sigma(\zeta)) d\pi_{v_{\Rtheta(s)},\sigma_s(\xi),\sigma(\zeta)} (\nu^{[y_0]})(z) \right) d\nu(y) \right) dQ^{(\mu,x,t,\tau, \omega,\eta)} _{1,3,4,5,6} (\nu,s, v_s, \xi,\zeta)=
\end{equation*}
\begin{equation*}
\int  \left( \sum_{i,j} \int  f(z,\Rtheta(s),\sigma_s(\xi),\sigma(\zeta)) d\pi_{v_{\Rtheta(s)},\sigma_s(\xi),\sigma(\zeta)} (\nu^{[(i,j)]})(z)  \nu([(i,j)]) \right)  dQ^{(\mu,x,t,\tau, \omega,\eta)} _{1,3,4,5,6}  (\nu,s, v_s, \xi,\zeta)
\end{equation*}
\begin{equation*}
\begin{split}
   & = \int   \left(   \sum_{i,j} \int f(z,\Rtheta(s),\sigma_s(\xi),\sigma(\zeta)) d\pi_{v_{\Rtheta(s)},\sigma_s(\xi),\sigma(\zeta)}  ( \frac{\sigma \nu|_{[(i,j)]}}{\nu([i,j]})(z) \right) dQ^{(\mu,x,t,\tau, \omega,\eta)} _{1,3,4,5,6}  (\nu,s, v_s, \xi,\zeta).
\end{split}
\end{equation*}
As generically we have $\dim \nu \geq \gamma >0$ in the above integral (since we are working with an ergodic component with positive entropy), we have by Lemma 5.1 (for $k=1$, so $U_s=\Phi_s$),
\begin{equation*}
\begin{split}
   & = \int \left( \sum_{i,j} \int f(z,\Rtheta(s),\sigma_t(\xi),\sigma(\zeta)) d \Phi_s \frac{ (\pi_{v_s,\xi,\zeta} \nu)|_{\mathcal{A}_t ^1 (\pi_{v_s,\xi,\zeta} ([(i,j)]))}}{\pi_{v_s,\xi,\zeta} \nu(\mathcal{A}_t ^1 (\pi_{v_s,\xi,\zeta} ([(i,j)]))} (z) \right) dQ^{(\mu,x,t,\tau, \omega,\eta)} _{1,3,4,5,6} (\nu,s, v_s, \xi,\zeta)
   \end{split}
\end{equation*} 
where by the notation $\mathcal{A}_t ^1 (\pi_{v_s,\xi,\zeta} ([(i,j)]))$ we mean the unique partition element of $\mathcal{D}_{m_1 ^{r_1(t)}} \times \mathcal{D}_{m_2}$ that contains all elements $\pi_{v_s,\xi,\zeta} (y)$ for all $y\in [(i,j)]$ (except for maybe on the measure zero boundary of the cell).  Changing variables,
\begin{equation*}
\begin{split}
 & = \int \left( \sum_{i,j}  \int f(\Phi_s (z),\Rtheta(s),\sigma_t(\xi),\sigma(\zeta)) d\frac{ (\pi_{v_s,\xi,\zeta} \nu)|_{\mathcal{A}_t ^1 (\pi_{v_s,\xi,\zeta} [(i,j)])}}{\pi_{v_s,\xi,\zeta} \nu(\mathcal{A}_t ^1 (\pi_{v_s,\xi,\zeta} [(i,j)])} (z) \right) dQ^{(\mu,x,t,\tau, \omega,\eta)} _{1,3,4,5,6} (\nu,s, v_s, \xi,\zeta)
\end{split}
\end{equation*}
\begin{equation*}
\begin{split}
 & = \int \left(  \sum_{i,j} \int f\circ U(z,s,\xi,\zeta) d\frac{ (\pi_{v_s,\xi,\zeta} \nu)|_{\mathcal{A}_t ^1 (\pi_{v_s,\xi,\zeta} [(i,j)])}}{\pi_{v_s,\xi,\zeta} \nu(\mathcal{A}_t ^1 (\pi_{v_s,\xi,\zeta} [(i,j)])} (z) \right) dQ^{(\mu,x,t,\tau, \omega,\eta)} _{1,3,4,5,6} (\nu,s, v_s, \xi,\zeta).
\end{split}
\end{equation*}
Finally, by retracing our steps, and using \eqref{Equation 5.3}, we see that
\begin{equation*}
\begin{split}
 & = \int f\circ U d\nu^{(\mu,x,t,\tau, \omega,\eta)}
\end{split}
\end{equation*}
\end{proof}

\subsection{Some properties of our U invariant measures}
Fix an element $(\mu_0,x_0,t_0,\tau_0, \omega_0,\eta_0)\in E_\gamma$ s.t. $Q^{(\mu_0,x_0,t_0,\tau_0, \omega_0,\eta_0)} _{1,2}$ is an ergodic CP distribution of dimension $\geq \gamma$, and 
\begin{equation*}
\nu_\infty = \nu^{(\mu_0,x_0,t_0, \tau_0, \omega_0,\eta_0)} = \int \left( \pi_{v_t,\omega,\eta} \mu \times \delta_t \times \delta_\omega \times \delta_\eta \right) dQ ^{(\mu_0,x_0,t_0,\tau_0, \omega_0,\eta_0)} _{1,3,4,5,6} (\mu,t,v_t, \omega,\eta).
\end{equation*}
is $U$ invariant.

We record some other useful properties of the measure $\nu_\infty$ and the partitions $\mathcal{B}_k$, defined in \eqref{Equation 5.1}. We defer the proof to section \ref{Section Remaining proofs}.
\begin{Proposition} \label{Proposition properties of nu}
\begin{enumerate}
\item The partitions $\mathcal{B}_k$ generate the Borel sigma algebra of $[0,1]^2 \times \mathbb{T} \times [n_1]^\mathbb{N} \times [n_2] ^\mathbb{N}$.

\item For every $k\in \mathbb{N}$ and every element $B\in \mathcal{B}_k$ we have $\nu_\infty (\partial B)=0$.

\item Suppose that the $(\omega_0, \eta_0)$ in the assumption of Proposition \ref{Lemma CP chain} are $\alpha_1\times \alpha_2$ typical for $t_0$ in the sense of Theorem \ref{Theorem X ergodic MPS}, for some Bernoulli measures $\alpha_1 \in P([n_1]^\mathbb{N})$ and $\alpha_2\in P([n_2]^\mathbb{N})$. Then we may assume the marginal of $\nu_\infty$ on the third component and fourth component gives full measure to the set of $\alpha_1\times \alpha_2$ generic points, with respect to the product system $([n_1]^\mathbb{N}\times [n_2]^\mathbb{N}, \sigma\times \sigma,\alpha_1\times \alpha_2)$.

\item Entropy-wise, we have $0<h(\nu_\infty, U) \leq \log | \mathcal{B}_1|$.
\end{enumerate}
\end{Proposition}

We next outline another important property of measure $\nu_\infty$. By applying Proposition \ref{Proposition 3.7} to the ergodic CP distribution $Q_{1,2} ^{(\mu_0,x_0,t_0, \tau_0, \omega_0,\eta_0)}$ we see that: for any $\epsilon>0$ there exits $k_0 (\epsilon)$ s.t.  for $Q^{(\mu_0,x_0,t_0, \tau_0, \omega_0,\eta_0)}$ a.e. $\mu$ and $\mu$ a.e.  $x$ we have
\begin{equation} \label{Equation 5.14}
\begin{aligned}
{} & \liminf_{N} \frac{1}{N} | \lbrace 1\leq k\leq N : \max_{ |u|=k_0 (\epsilon)} \mu^{[x_0 ^{k-1}]} ([u]) \le \epsilon \text { and } \\
& H( \mu^{[x_0 ^{k-1}]} , \mathcal{J}_1 ^p \times \mathcal{J}_2 ^p) \geq p \cdot (\gamma \log m_2 - \epsilon) \rbrace | > 1- 2\epsilon, \text{ for any } p\geq k_0 (\epsilon) \\
\end{aligned}
\end{equation}
By applying part (3) of Lemma 4.1, we see that for any $\epsilon>0$ there is some $\delta(\epsilon)>0$ and some $k_1 (\epsilon) \in \mathbb{N}$ s.t. for $Q_1 ^{(\mu_0,x_0,t_0, \tau_0, \omega_0,\eta_0)}$ a.e. $\mu$ and $\pi_{v_t,\omega,\eta} \mu$ a.e. $z$ we have 
\begin{equation} \label{Equation 5.15}
\begin{aligned}
{} & \liminf_{N} \frac{1}{N} | \lbrace 1\leq k\leq N : \sup_{y\in [0,1]^2} (\pi_{v_t,\omega,\eta} \mu)^{\mathcal{A}_k ^t (z)} (B(y,\delta)) \leq \epsilon \text { and } \\
& H( (\pi_{v_t,\omega,\eta} \mu)^{\mathcal{A}_k ^t (z)} , \mathcal{D}_{2^p}\times \mathcal{D}_{2^p}) \geq p \cdot (\gamma \log 2 - 2\epsilon) \rbrace | > 1- 2\epsilon, \text{ for any } p\geq k_1 (\epsilon) \\
\end{aligned}
\end{equation}

In particular, the above is true for $Q_{1,3,4,5,6} ^{(\mu_0,x_0,t_0, \tau_0, \omega_0,\eta_0)}$ a.e. $(\mu,t, v_t,\omega,\eta)$ and $\pi_{v_t,\omega,\eta}$ a.e. $z$. On the other hand, by the definition of the measure $\nu_\infty$, selecting $(z,t,\omega,\eta)$ according to $\nu_\infty$ can be done by first drawing $(\mu,t,v_t, \omega,\eta)$ according to $Q_{1,3,4,5,6} ^{(\mu_0,x_0,t_0, \tau_0, \omega_0,\eta_0)}$ and then selecting $z$ according to $\pi_{v_t,\omega,\eta} \mu$.  Thus, we have the following Proposition:

\begin{Proposition} \label{Proposition 5.7}
The measure $\nu_\infty$ satisfies the following property: for every $\epsilon>0$ there exists $\delta(\epsilon)>0$ and $k_1 (\epsilon)$ s.t. for $\nu_\infty$ a.e. $(z,t,\omega,\eta)$  there exists a measure $\mu \in P(X)$ s.t. 
\begin{enumerate}
\item $\pi_{v_t,\omega,\eta} \mu \in P(\ell\cap [ \Fomega\times \Eeta ])$ for some line $\ell$ with slope $m_1 ^t$.

\item  \eqref{Equation 5.15} holds for $\pi_{v_t,\omega,\eta}$ and for $z$. 
\end{enumerate}
\end{Proposition}

Let us sum up the results of this section:

\begin{theorem} \label{Theorem 5.10}
There exists an ergodic $U$-invariant measure $\nu_{\infty}$ s.t. $h(\nu_{\infty},U)>0$, and $\nu_{\infty}$ satisfies the properties stated in Propositions \ref{Proposition 5.7} and \ref{Proposition properties of nu}.
\end{theorem}

Notice that if $\nu_\infty$ is not ergodic, we may move to an ergodic component of $\nu_\infty$. To get positive entropy, we use the well known fact that the entropy of $\nu_\infty$ is the average over the entropies of its ergodic components.

\section{Proof of Theorem \ref{Theorem main Theorem}}
\subsection{An application of the Sinai factor Theorem}
Recall that a sequence $\lbrace x_k \rbrace_{k\in \mathbb{N}} \subset \mathbb{T}$ is uniformly distributed (UD) if for every sub-interval $J\subseteq \mathbb{T}$ we have
\begin{equation*}
\frac{1}{N} | \lbrace 0\ \leq k \leq N-1: x_k \in J \rbrace| \rightarrow \lambda (J)
\end{equation*}

In \cite{wu2016proof}, Wu was able to prove the following Theorem by using the Sinai factor Theorem.
\begin{theorem} \label{Theorem 6.1} (\cite{wu2016proof}, Theorem 6.1)
Let $(X,T,\mu)$ be an ergodic measure preserving system with $h(\mu,T)>0$. Let $\mathcal{A}$ be a generator with finite cardinality, and let $\lbrace \mathcal{A}_k \rbrace_k$ denote the filtration generated by $\mathcal{A}$ and $T$. Suppose that $\mu (\partial A)=0$ for every $k\in \mathbb{N}$ and every $A\in \mathcal{A}_k$. Let $\beta \notin \mathbb{Q}$.

Then for any $\epsilon>0$ there exists $k_2 = k_2 (\epsilon)>0$ s.t. for all $k\geq k_2$ there exists a disjoint family of measurable sets $\lbrace C_i \rbrace_{i=1} ^{N(k,\epsilon)}, C_i \subset X,$ such that:
\begin{enumerate}
\item $\mu( \bigcup C_i) > 1-\epsilon.$

\item For every $1 \leq i \leq N(k,\epsilon)$, $| \lbrace A\in \mathcal{A}_k: C_i\cap A \rbrace| \leq e^{k\cdot \epsilon}$.

\item There exists another  disjoint family of measurable sets $\lbrace \tilde{C}_i \rbrace_{i=1} ^{N(k,\epsilon)}, \tilde{C}_i \subset X$, s.t. for every $1\leq i \leq N(k,\epsilon)$ we have:
\begin{itemize}
\item $C_i \subseteq \tilde{C}_i$

\item $\mu (C_i) \geq (1-\epsilon) \mu (\tilde{C}_i)$.

\item for $\mu$ a.e. $x$ we have that the sequence 
\begin{equation*}
\lbrace R_\beta ^k (0) \in \mathbb{T} : k\in \mathbb{N} \text{ and }T^k (x)\in \tilde{C}_i \rbrace
\end{equation*}
is UD.
\end{itemize}
\end{enumerate}
\end{theorem}

\subsection{Extracting geometric information from Theorem \ref{Theorem 6.1}}
The following Proposition is modeled after Proposition 7.1 in \cite{wu2016proof}. As in \cite{wu2016proof}, we denote the coordinate projections\footnote{For example $\Pi_1 (z,t,\omega,\eta)=z$.} of $[0,1]^2 \times \mathbb{T} \times [n_1]^\mathbb{N} \times[n_2] ^\mathbb{N}$ by $\Pi_i$, for $i=1,2,3,4$, and similarly $\Pi_{1,2}$ and $\Pi_{3,4}$.


\begin{Proposition} \label{Proposition 7.1}
There exists a constant $C>0$ s.t. for every $\epsilon>0$ there is some $r_0 = r_0 (\epsilon)$ and $k_5 = k_5 (\epsilon)\in \mathbb{N}$ s.t. for every $k\geq k_5 (\epsilon)$ the following is true:

For $\nu_\infty$ a.e. $(z,t,\omega,\eta)$ we can find a measure $\nu \in P([0,1]^2)$, a measurable set $D \subset [0,1]^2 \times \mathbb{T} \times [n_1]^\mathbb{N} \times[n_2] ^\mathbb{N}$, and a set $\mathcal{N} \subset \mathbb{N}$ such that:
\begin{enumerate}
\item $\nu \in P(\ell \cap [ \Fomega\times \Eeta ])$ for some line $\ell$ with slope $m_1 ^t$.

\item We have
\begin{equation*}
 \frac{1}{k} \log N_{2^{-k}} (D_1) \leq C\cdot (\epsilon+\frac{1}{k}), 
\end{equation*}
where $\Pi_1 (D)= D_1$, and $N_{2^{-k}} (A)$ is the number of $k$-level dyadic boxes $A$ intersects. In addition, 
\begin{equation*}
\frac{1}{k \log 2} \log N_{2^{-k}} (\pi_{m_1} \times \pi_{m_2} (\tilde{D}_{3,4}))  \leq C\cdot (\epsilon+\frac{1}{k}) + \max_{(i,j) \in [n_1] \times [n_2]} \frac{\log |\Gamma_i|}{\log m_1} + \frac{\log |\Lambda_j|}{\log m_2}.
\end{equation*}
where $\Pi_{3,4} (D)= D_{3,4}$ and $\tilde{D}_{3,4} = \bigcup_{(\omega,\eta)\in D_{3,4}} \tilde{F}_\omega \times \tilde{E}_\eta \subset [m_1]^\mathbb{N} \times [m_2]^\mathbb{N}$. 

\item For every $p\in \mathcal{N}$ we have $U^p (z,t,\omega,\eta)\in D$.

\item $\lambda( \overline{\lbrace \Rtheta^p (t):p\in \mathcal{N} \rbrace} ) \geq 1 - C \cdot \epsilon$.

\item For every $p\in \mathcal{N}$,
\begin{equation*}
\inf_{y\in \mathbb{R}^2} \frac{1}{k\log 2}  H( \nu^{\mathcal{A}_p ^t (z)} |_{B(y,r_0)^C} , \mathcal{D}_{2^k}) \geq \gamma - C\cdot \sqrt{\epsilon}.
\end{equation*}

\item Suppose $(\omega_0,\eta_0)$ from the condition in the statement of Proposition \ref{Lemma CP chain} are typical with respect to $t_0$ and a product of Bernoulli measures $\alpha_1\times \alpha_2 \in P([n_1]^\mathbb{N} \times[n_2] ^\mathbb{N})$, in the sense of Theorem \ref{Theorem X ergodic MPS}. Then for $\nu_\infty$ a.e. $(z,t,\omega,\eta)$ we can construct sets and measures with all the above properties, with the additional property that
\begin{equation*}
\frac{1}{k \log 2} \log N_{2^{-k}} (\pi_{m_1} \times \pi_{m_2} (\tilde{D'}_{3,4} ))  \leq C\cdot (\epsilon +\frac{1}{k}) + \sum_{i=0} ^{n_1-1}  \alpha_1([i])\cdot \frac{\log |\Gamma_{i}|}{\log m_1} + \sum_{j=0} ^{n_2-1}  \alpha_2([i])\cdot \frac{\log |\Lambda_{j}|}{\log m_2} ,
\end{equation*}
where $D_{3,4} ' = \bigcup_{k\in \mathcal{N}} \Pi_{3,4} U^k (z,t,\omega,\eta)$ and $\tilde{D'}_{3,4} = \bigcup_{(\omega,\eta)\in D_{3,4} '} \tilde{F}_\omega \times \tilde{E}_\eta.$
\end{enumerate}
\end{Proposition} 

We shall require two Lemmas for the proof. Both can be found in \cite{wu2016proof}. For a set $O \subset \mathbb{N}$ we denote the density of $O$ in $\mathbb{N}$ by
\begin{equation} \label{Eq density}
d(O,\mathbb{N}): = \lim_N \frac{1}{N} |O\cap [0,N-1]|
\end{equation}
If the limit does not exists we call the $\limsup$ the upper density of $O$ in $\mathbb{N}$ which we denote by $\overline{d}(O,\mathbb{N})$, and the $\liminf$ the lower density of $O$ in $\mathbb{N}$, denoted by $\underline{d}(O,\mathbb{N})$.
\begin{Lemma} (\cite{wu2016proof}, Lemma 7.2) \label{Lemma 7.2}
Let $\lbrace x_k \rbrace \subset \mathbb{T}$ be UD. Let $O\subseteq \mathbb{N}$. Then
\begin{equation*}
\lambda( \overline{\lbrace x_k : x_k \in O\rbrace } ) \geq \overline{d} (O,\mathbb{N}).
\end{equation*}
\end{Lemma}

\begin{Lemma} (\cite{wu2016proof}, Lemma 7.3) \label{Lemma 7.3}
Let $\mu \in P(\mathbb{R}^d)$ and fix $0<\delta<1$. If $\sup_{y\in \mathbb{R}^d} \mu (B(y,\delta)) \leq \epsilon$ then for all $k\in \mathbb{N}$ such that $2^{-k} \leq \delta$ we have
\begin{equation*}
\inf_{y\in \mathbb{R}^2}   H( \mu |_{B(y,r_0)^C} , \mathcal{D}_{2^k}) \geq  H(\mu,\mathcal{D}_{2^k}) - C_1 \cdot k \sqrt{\epsilon}.
\end{equation*}
For some constant $C_1$ that depends only on $d$.
\end{Lemma}

We now prove Proposition 7.1, under the additional assumption that $(\omega_0,\eta_0)$ from Lemma \ref{Lemma CP chain} are typical with respect to a product of Bernoulli measures measures $\alpha_1\times \alpha_2$ and $t_0 \in \mathbb{T}$, in the sense of part (6). If this is not the case then proof follows along the same lines, and  is actually easier.

Let $\epsilon>0$.

\textbf{Choice of the integer $k_5$ and $r_0$} By Theorem \ref{Theorem 5.10}, $\nu_\infty$ is ergodic, has positive entropy and satisfies  Proposition \ref{Proposition 5.7}. Put $r_0 (\epsilon) := \delta (\epsilon)$, where $\delta (\epsilon)$ is the number from  Proposition \ref{Proposition 5.7}. Recall the partition $\mathcal{B}_1$ of $[0,1]^2 \times \mathbb{T} \times [n_1]^\mathbb{N} \times [n_2] ^\mathbb{N}$, defined in \eqref{Equation 5.1}. Recall that $\mathcal{B}_1$ is a partition of finite cardinality, and that by Proposition \ref{Proposition properties of nu}, $\nu_{\infty} (\partial B)=0$, for all $B\in \mathcal{B}_k$ and all $k\geq 1$. We may thus apply Theorem \ref{Theorem 6.1} to the  dynamical system $([0,1]^2 \times \mathbb{T} \times [n_1] ^\mathbb{N} \times [n_2] ^\mathbb{N}, U, \nu_\infty)$.

In addition, for every $i\in [n_1]$ define continuous functions $f_i:[0,1]^2 \times \mathbb{T} \times [n_1]^\mathbb{N} \times [n_2] ^\mathbb{N}\rightarrow \mathbb{R}$ by $f_i (z,t,\omega,\eta) = 1_{\lbrace \omega_1 = i \rbrace} (\omega)$. Let $f_i ^k$ be the ergodic average (with respect to $([n_1]^\mathbb{N}, \sigma)$), 
\begin{equation*}
f_i ^k (z,t,\omega,\eta)= \frac{1}{[\theta\cdot  k]} \sum_{p=0} ^{[\theta\cdot k]-1} f_i (\sigma^p (\omega)).
\end{equation*} 
Similarly, for every $j\in [n_2]$ define continuous functions $g_j:[0,1]^2 \times \mathbb{T} \times [n_1]^\mathbb{N}\times [n_2] ^\mathbb{N}\rightarrow \mathbb{R}$ by $g_j (z,t,\omega,\eta) = 1_{\lbrace \eta_1 = j \rbrace} (\omega)$. Let $g_j ^k$ be the ergodic average (with respect to $([n_2]^\mathbb{N}, \sigma)$),
\begin{equation*}
g_j ^k (z,t,\omega,\eta)= \frac{1}{k} \sum_{p=0} ^{k-1} g_j (\sigma^p (\eta)).
\end{equation*}

By Proposition \ref{Proposition properties of nu}, $\Pi_{3,4} \nu_\infty$ almost every $(\omega,\eta)$ is generic with respect to the  product system $([n_1]^\mathbb{N} \times [n_2]^\mathbb{N}, \sigma\times \sigma, \alpha_1\times \alpha_2)$. So for $\nu_\infty$ a.e. $(z,t,\omega,\eta)$, $\omega$ is generic for $([n_2]^\mathbb{N}, \sigma,\alpha_1)$ and $\eta$ is generic for $([n_2]^\mathbb{N}, \sigma,\alpha_2)$, and therefore for every $(i,j)\in [n_1]\times [n_2]$,
\begin{equation*}
\lim_k f_i ^k (z,t,\omega,\eta) = \int 1_{\lbrace  \omega_1 = i \rbrace} (\omega) d \alpha_1 (\omega) = \alpha_1([i]), \quad \text{and similarly } \lim_k g_j ^k (z,t,\omega,\eta) = \alpha_2 ([j]).
\end{equation*}
Thus, by $n_1 \cdot n_2$ applications of Egorov's Theorem, we may find an integer $k_3 (\epsilon)$ such that 
\begin{equation} \label{Eq freq}
\begin{aligned}
{} & V = \lbrace (z,t,\omega,\eta): \forall k>k_3 (\epsilon), \quad \forall (i,j)\in [n_1]\times [n_2],  \\
& |f_i ^k (z,t,\omega,\eta) - \alpha_1([i])| < \frac{\epsilon}{n_1 \cdot n_2}, \quad \text{ and } |g_j ^k (z,t,\omega,\eta) - \alpha_2([j])| <\frac{\epsilon}{n_1\cdot  n_2} \rbrace \\
\end{aligned}
\end{equation}
has measure $\nu_\infty (V) \geq 1-\epsilon$.

Let $k_2 (\epsilon)$ be the integer provided by Theorem \ref{Theorem 6.1}. Let $k_1 (\epsilon)$ be the integer from Proposition \ref{Proposition 5.7}. Let $k_4 (\epsilon)$ be such that $2^{-k} \leq \delta$ for all $k\geq k_4$. Let 
\begin{equation*}
k_5 = \max \lbrace k_1 (\epsilon), k_2 (\epsilon)\cdot \frac{\log m_2}{\log 2} , k_3 (\epsilon)\cdot \frac{\log m_2}{\log 2}, k_4 (\epsilon) \rbrace.
\end{equation*}
We will show that $k_5$ can be taken to be the integer promised in the statement of Proposition \ref{Proposition 7.1}.

\textbf{Construction of the sets $\mathcal{N}$ and $D$} Let $k\geq k_5$. Define $\tilde{k} = [ k \cdot \frac{\log 2}{\log m_2}]+1$. Then $\tilde{k} \geq k_2 (\epsilon), k_3 (\epsilon)$.  By Theorem \ref{Theorem 6.1} we can find  disjoint families of measurable sets $$\lbrace C_i \rbrace_{i=1} ^{N(\tilde{k},\epsilon)}, \quad \lbrace \tilde{C}_i \rbrace_{i=1} ^{N(\tilde{k},\epsilon)}, \text{ such that all sets are subsets of }  [0,1]^2 \times \mathbb{T} \times [n_1]^\mathbb{N} \times [n_2] ^\mathbb{N}$$  satisfying the conditions of Theorem \ref{Theorem 6.1} with respect to the partition $\mathcal{B}_{\tilde{k}}$.

 Let $A' \subset [0,1]^2 \times \mathbb{T} \times [n_1]^\mathbb{N} \times [n_2] ^\mathbb{N}$ denote the set of $(z,t,\omega,\eta)$ such that:
\begin{itemize}
\item  The sequence
\begin{equation} 
\lbrace R_\theta ^k (t) \in \mathbb{T} : k\in \mathbb{N} \text{ and }U^k (z,t,\omega,\eta)\in \tilde{C}_i \rbrace
\end{equation}
is UD for every $1\leq i \leq N(\tilde{k},\epsilon)$.

\item There exists a measure $\mu =\mu_{z,t,\omega,\eta}$ such that $\pi_{v_t,\omega,\eta} \mu \in P(\ell \cap [ \Fomega \times \Eeta])$ for some line $\ell$ with slope $m_1 ^t$, and \eqref{Equation 5.15} holds for $\pi_{v_t,\omega.\eta} \mu$ and $z$.
\end{itemize}
By Theorem \ref{Theorem 6.1} part (3), and by Proposition \ref{Proposition 5.7}, since $k\geq k_1(\epsilon)$ and by the choice of $r_0(\epsilon)$, $\nu_\infty (A')=1$.

Next, for $(z,t,\omega,\eta)$ and $1\leq i \leq N(\tilde{k},\epsilon)$ define the sequences of visiting times
\begin{eqnarray*}
B(C_i, z, t, \omega,\eta) &=& \lbrace k\in \mathbb{N}: \quad U^k (z,t,\omega,\eta) \in C_i \rbrace\\
B(\tilde{C}_i, z, t, \omega,\eta)&=& \lbrace k\in \mathbb{N}:\quad U^k (z,t,\omega,\eta) \in \tilde{C}_i \rbrace\\
B(V,z,t,\omega,\eta)& =& \lbrace k\in \mathbb{N}:\quad  U^k (z,t,\omega,\eta) \in V \rbrace.
\end{eqnarray*}
Recall the definition of the density of a set of integers from \eqref{Eq density}. Let $A''$ be the set of all $(z,t,\omega,\eta)$ such that for all $i$
\begin{equation*}
d(B(C_i, z, t, \omega,\eta), \mathbb{N}) = \nu_\infty (C_i), \quad d(B(\tilde{C}_i, z, t, \omega,\eta), \mathbb{N}) = \nu_\infty (\tilde{C}_i),
\end{equation*}
and
\begin{equation*}
d(B(V, z, t, \omega,\eta), \mathbb{N}) = \nu_\infty (V) \geq 1- \epsilon.
\end{equation*}
Then the ergodicity of $\nu_\infty$ implies that $\nu_\infty (A'')=1$. Let $A= A' \cap A''$, then $\nu_\infty (A)=1$.

Let $(z,t,\omega,\eta)\in A$. Then $(z,t,\omega,\eta)\in A'$, so there exists $\mu=\mu_{z,t,\omega,\eta}$ such that $$\pi_{v_t,\omega,\eta} \mu \in P(\ell \cap [ \Fomega \times \Eeta]) \text{ for some line } \ell \text { with slope } m_1 ^t ,$$and \eqref{Equation 5.15} holds for $\pi_{v_t,\omega,\eta} \mu$ and $z$. Denote $\nu = \pi_{v_t,\omega,\eta} \mu$. By the choice of $r_0 (\epsilon) = \delta(\epsilon)$ , as $k\geq k_1(\epsilon)$, and by \eqref{Equation 5.15}, the set
\begin{equation*}
A(\nu,z,t,\omega,\eta):= \lbrace  p\in \mathbb{N} : \sup_{y\in [0,1]^2} \nu^{\mathcal{A}_p ^t (z)} (B(y,r_0(\epsilon))) \leq \epsilon \text { and }  H( \nu^{\mathcal{A}_p ^t (z)} , \mathcal{D}_{2^k}) \geq k \cdot (\gamma \log m_2 - 2\epsilon) \rbrace  
\end{equation*}
has lower density $\geq 1-2\epsilon$ in $\mathbb{N}$.

 Since $d(B(V, z, t, \omega,\eta), \mathbb{N}) \geq 1- \epsilon$, it follows, by the inclusion-exclusion principle, that the set
\begin{equation*}
B(V, z, t, \omega,\eta) \cap A(\nu,z,t,\omega,\eta)
\end{equation*}
has lower density at least $1-3\cdot \epsilon$ in $\mathbb{N}$.

On the other hand, by Theorem \ref{Theorem 6.1} part (1), the density of $\bigcup_{i=1} ^{N(\tilde{k},\epsilon)} B(C_i,z,t,\omega,\eta)$ in $\mathbb{N}$ is at least $1-\epsilon$. Notice that the sets $B(C_i,z,t,\omega,\eta)$ are disjoint. It follows that there exists at least one $1\leq i_0 \leq N(\tilde{k},\epsilon)$ such that\footnote{If $S_1,S_2 \subseteq \mathbb{N}$, we define $\underline{d}(S_1,S_2)= \liminf_N \frac{S_1\cap (S_2\cap[N])}{S_2\cap [N]}$.}
\begin{equation*}
\underline{d} (  \left(  B(V, z, t, \omega,\eta) \cap A(\nu,z,t,\omega,\eta) \right) \cap B(C_{i_0},z,t,\omega,\eta), \quad B(C_{i_0},z,t.\omega,\eta)) \geq 1- 4 \epsilon.
\end{equation*}
We thus set $D = C_{i_0}$ and $\mathcal{N} = A(\nu,z,t,\omega,\eta)\cap B(V, z, t, \omega,\eta) \cap B(C_{i_0},z,t,\omega,\eta)$.

\textbf{Proof of the Proposition \ref{Proposition 7.1} part (4)}
\begin{Lemma}
$\lambda ( \overline{\lbrace \Rtheta^k (t):k\in \mathcal{N} \rbrace} ) > 1-5\epsilon.$
\end{Lemma}

\begin{proof} Since $C_{i_0} \subset \tilde{C_{i_0}}$, by the choice of $A''$ and Theorem \ref{Theorem 6.1} part (3) we have
\begin{equation*}
d( B(C_{i_0},z,t,\omega,\eta) , B( \tilde{C}_{i_0},z,t,\omega,\eta)) = d( B(C_{i_0},z,t,\omega,\eta) , \mathbb{N}) \cdot d(B( \tilde{C}_{i_0},z,t,\omega,\eta), \mathbb{N})^{-1} \geq 1-\epsilon.
\end{equation*}
Therefore, 
\begin{equation*}
\underline{d} ( \mathcal{N}, B(\tilde{C}_{i_0},z,t,\omega,\eta)) \geq \underline{d} ( \mathcal{N}, B(C_{i_0},z,t,\omega,\eta)) \cdot \underline{d} ( B(C_{i_0},z,t,\omega,\eta), B(\tilde{C}_{i_0},z,t,\omega,\eta)   
\end{equation*}
\begin{equation*}
\begin{split}
& \geq (1- 4 \epsilon)\cdot (1-\epsilon).
\end{split}
\end{equation*}
Since $(z,t,\omega,\eta)\in A'$ then $\lbrace R_\theta ^k (t) \in \mathbb{T} : k\in B(\tilde{C}_{i_0},z,t,\eta) \rbrace$ is UD. It follows from Lemma \ref{Lemma 7.2} that
\begin{equation*}
\lambda( \overline{\lbrace \Rtheta^k (t):k\in \mathcal{N} \rbrace} ) \geq (1- 4 \epsilon)(1-\epsilon) > 1-5\epsilon.
\end{equation*}
\end{proof}

\textbf{Proof of Proposition \ref{Proposition 7.1} parts (2) and (6)}
\begin{Claim} \label{Claim projections}
 For $i=3,4$ let $\Pi_i (D) = D_i$ , and let $D_i ' = \bigcup_{k\in \mathcal{N}} \Pi_i (U^k (z,t,\omega,\eta))  \subset D_i$. Define
\begin{equation*}
\pi_{m_1} (\tilde{F}_{D_3'}): = \bigcup_{\xi\in D_3'} \pi_{m_1} (\tilde{F}_\xi), \quad \text{ and  } \pi_{m_2} (\tilde{E}_{D_4'}) := \bigcup_{\zeta\in D_4'} \pi_{m_2} (\tilde{E}_\zeta). 
\end{equation*}
Then, for some constant $C_2$ that does not depend on $k$ or $\epsilon$, 
\begin{equation} \label{Eq F}
\frac{\log N_{2^{-k}} (\pi_{m_1} (\tilde{F}_{D_3 '}))}{k \log 2} \leq \sum_{i=0} ^{n_1-1} \alpha_1([i]) \cdot \frac{\log |\Gamma_{i}|}{\log m_1}  +C_2 \cdot (\epsilon+\frac{1}{k})
\end{equation}
and
\begin{equation} \label{Eq E}
\frac{\log N_{2^{-k}} (\pi_{m_2} (\tilde{E}_{D_4 '}))}{k \log 2} \leq \sum_{j=0} ^{n_2-1} \alpha_2([j]) \cdot \frac{\log |\Lambda_{j}|}{\log m_2}  +C_2 \cdot (\epsilon+\frac{1}{k})
\end{equation}
\end{Claim}

\begin{proof}
We first study $D_3$. By Theorem \ref{Theorem 6.1} part (2) and the choice of $D$,
\begin{equation*}
| \lbrace B \in \mathcal{B}_{\tilde{k}}: D \cap B \neq \emptyset \rbrace | \leq e^{\epsilon\cdot \tilde{k}}. 
\end{equation*}
Let $t_0 \in \mathbb{T}$. Then, by the last displayed equation, the definition of the partitions $\mathcal{B}_k$, and recalling that $\mathcal{I}_1 ^p$ is the $p$-level cylinder partition of $[n_1]^\mathbb{N}$,
\begin{equation*}
|\lbrace I \in \mathcal{I}_1 ^{r_{\tilde{k}}(t_0)}: D_3 \cap I \neq \emptyset \rbrace| \leq |\bigcup_{t\in \mathcal{C}_{\tilde{k}}} \lbrace I \in \mathcal{I}_1 ^{r_{\tilde{k}}(t)}: D_3 \cap I \neq \emptyset \rbrace | \leq e^{\epsilon\cdot \tilde{k}}.
\end{equation*} 
By Claim \ref{Claim tau t}, we know that there exists some constant $C\in \mathbb{N}$ such that for every $t$
\begin{equation*}
r_{\tilde{k}} (t) \geq \tilde{k} \cdot  \theta -C \geq [\tilde{k} \cdot  \theta] -C. 
\end{equation*}
It follows that
\begin{equation*}
N(D_3, \mathcal{I}_1 ^{[\tilde{k} \cdot  \theta]-C}): = | \lbrace I \in  \mathcal{I}_1 ^{[\tilde{k} \cdot  \theta] -C}: D_3 \cap I \neq \emptyset \rbrace | \leq |\lbrace I \in \mathcal{I}_1 ^{r_{\tilde{k}}(t_0)}: D_3 \cap I \neq \emptyset \rbrace| \leq e^{\epsilon\cdot \tilde{k}}.
\end{equation*}
Since $\tilde{k} = [k \frac{\log 2}{\log m_2}]+1$ we see that, since $\theta = \frac{\log m_2}{\log m_1}$,
\begin{equation*}
[\tilde{k} \cdot  \theta] \geq k\frac{\log 2}{\log m_1} - 1 \geq k\frac{\log 2}{\log m_1} -C \geq [k\frac{\log 2}{\log m_1}]-C,
\end{equation*}
since $C>1$. 

Now, recall that $D_3 ' = \bigcup_{k\in \mathcal{N}} \Pi_3 (U^k (z,t,\omega,\eta))  \subset D_3$. Then, as $\tilde{k}\geq k_3 (\epsilon)$, by the definition of $B(V,z,t,\omega,\eta)$ and of the set $V$ (recall \eqref{Eq freq}), we have, for every $\xi \in D_3 '$ and every $0\leq i \leq n-1$,
\begin{equation*}
|\lbrace |1\leq j \leq [k\frac{\log 2}{\log m_1}] - 2C: \xi_j=i \rbrace| \leq |\lbrace |1\leq j \leq [\tilde{k}\cdot \theta]: \xi_j=i \rbrace| \leq [\theta \cdot \tilde{k}]\cdot \alpha_1([i])+[\theta \cdot \tilde{k}]\cdot \epsilon.
\end{equation*}

We can now calculate. Define for every $\xi \in D_3 '$ 
\begin{equation*}
\widehat{\pi_{m_1} (\tilde{F}_\xi)} := \lbrace \sum_{i=1} ^{[k \frac{\log 2}{\log m_1}]-2C} \frac{x_i}{m_1 ^i}: x_i \in \Gamma_{\xi_i} \rbrace, \quad \text{ and } \widehat{ \pi_m(\tilde{F}_{D_3 '}) } : = \bigcup_{\xi \in D_3 '} \widehat{\pi_{m_1} (\tilde{F}_\xi)} .
\end{equation*}
Notice that $\widehat{ \pi_m(\tilde{F}_{D_3 '}) }$ is actually a finite union of sets of the form $\widehat{\pi_{m_1} (\tilde{F}_\xi)}$, and that, by considering $m_1$-adic rationals,
\begin{equation*}
N( \pi_{m_1} (\tilde{F}_{D_3 '}), m_1 ^{[k \frac{\log 2}{\log m_1}]-2C}) \leq 3\cdot N( \widehat{\pi_{m_1} (\tilde{F}_{D_3 '}}), m_1 ^{[k \frac{\log 2}{\log m_1}]-2C}).
\end{equation*}
Recall that for a set  $A\subseteq \mathbb{R}$, $N(A,p)$ denotes the number of $p$-adic intervals $A$ intersects, and that $N_{2^{-k}} (A):=N(A,2^k)$. Let $C_1 \in \mathbb{N}$ be such that $2^{C_1} >m_1$. Then,
\begin{eqnarray*}
N_{2^{-k}} (\pi_{m_1} (\tilde{F}_{D_3 '})) & \leq & 2^{C_1 +1}\cdot N( (\pi_{m_1} (\tilde{F}_{D_3 '}), m_1 ^{[k \frac{\log 2}{\log m_1}]}) \\
& \leq & 2^{C_1 +1}\cdot N( (\pi_{m_1} (\tilde{F}_{D_3 '}), m_1 ^{[k \frac{\log 2}{\log m_1}]-2C})\cdot m_1 ^{2C} \\
&\leq& 2^{C_1 +1}\cdot  3\cdot N( \widehat{\pi_{m_1} (\tilde{F}_{D_3 '})}, m_1 ^{[k \frac{\log 2}{\log m_1}]-2C}) \cdot m_1^{2C}\\
&\leq& 2^{C_1 +1}\cdot 3 \cdot   N(D_3, \mathcal{I}_1 ^{[k \frac{\log 2}{\log m_1}]-2C}) \cdot  \max_{\xi \in D_3 '} N(\widehat{\pi_{m_1} (\tilde{F}_\xi)} ,m_1 ^{[k \frac{\log 2}{\log m_1}]-2C}]) \cdot m_1^{2C} \\
 &\leq & 2^{C_1 +1}\cdot  3\cdot N(D_3, \mathcal{I}_1 ^{[\tilde{k} \cdot  \theta]-C}) \cdot \max_{\xi \in D_3} \prod_{i=0} ^{n_1-1} |\Gamma_{i}|^{|1\leq j \leq [k\frac{\log 2}{\log m_1}] - 2C: \xi_j=i \rbrace|} \cdot m_1^{2C}\\
 &\leq& 2^{C_1 +1}\cdot  3\cdot  e^{\epsilon\cdot \tilde{k}} \cdot \prod_{i=0} ^{n_1-1} |\Gamma_{i}|^{[\theta \cdot \tilde{k}]\cdot \alpha_1([i])+[\theta \cdot \tilde{k}]\cdot \epsilon} \cdot m_1^{2C}
\end{eqnarray*}
Taking $\log$ and dividing by $k\log 2$, recalling that $\theta = \frac{\log m_2}{\log m_1}$, yields \eqref{Eq F}.

For $D_4$, we follow a similar argument.  The main difference is that for $D_4$, by the definition of $\mathcal{B}_{\tilde{k}}$,
\begin{equation*}
| \lbrace J \in \mathcal{I}_2 ^{\tilde{k}}: D_4 \cap J \neq \emptyset \rbrace | \leq e^{\epsilon\cdot \tilde{k}}.
\end{equation*} 
since $\Pi_4 \mathcal{B}_{\tilde{k}} = \mathcal{I}_2 ^{\tilde{k}}$ - the cylinder partition of generation $\tilde{k}$. In addition, as $\tilde{k}\geq k_3 (\epsilon)$, by the definition of $B(V,z,t,\omega,\eta)$ and of the set $V$ (recall \eqref{Eq freq}), we have, for every $\zeta \in D_4 '$ and every $0\leq i \leq n_2-1$,
\begin{equation*}
|\lbrace |1\leq j \leq \tilde{k}: \zeta_j=i \rbrace| \leq \tilde{k} \cdot \alpha_2([i])+ \tilde{k}\cdot \epsilon.
\end{equation*}
Thus, by a similar argument to the one proving \eqref{Eq F}, we see that
\begin{equation*}
N_{2^{-k}} (\pi_{m_2} (\tilde{F}_{D_4 '}))  \leq 2^{C_1 +1}\cdot  3\cdot e^{\tilde{k}\cdot \epsilon}\cdot  \prod_{i=0} ^{n_2-1} |\Lambda_{i}|^{ \tilde{k}\cdot \alpha([i])+ \tilde{k}\cdot \epsilon} \cdot m_2
\end{equation*}
Taking $\log$ and dividing by $k\log 2$, this yields \eqref{Eq E}.
\end{proof}

Recall that we want to bound
\begin{equation*}
\frac{1}{k \log 2} \log N_{2^{-k}} (\pi_{m_1} \times \pi_{m_2} (\tilde{D'}_{3,4} ))
\end{equation*}
where $D_{3,4} ' = \bigcup_{k\in \mathcal{N}} \Pi_{3,4} U^k (z,t,\omega,\eta)$ and $\tilde{D'}_{3,4} = \bigcup_{(\omega,\eta)\in D_{3,4} '} \tilde{F}_\omega \times \tilde{E}_\eta$
and
\begin{equation*}
\pi_{m_1} \times \pi_{m_2} (\tilde{D'}_{3,4}) = \bigcup_{(\omega,\eta)\in D_{3,4}'} \pi_{m_1} (\tilde{F}_\omega) \times \pi_{m_2}(\tilde{E}_\eta).
\end{equation*} 
It follows by definition that $\pi_{m_1} \times \pi_{m_2} (\tilde{D'}_{3,4}) \subseteq \pi_{m_1} (\tilde{F}_{D_3 '}) \times \pi_{m_2} (\tilde{E}_{D_4 '})$. Thus,
\begin{equation*}
\frac{1}{k \log 2} \log N_{2^{-k}} (\pi_{m_1} \times \pi_{m_2} (\tilde{D'}_{3,4} )) \leq  \frac{1}{k \log 2} \log N_{2^{-k}} (\pi_{m_1} (\tilde{F}_{D_3 '})) +\frac{1}{k \log 2} \log N_{2^{-k}} (\pi_{m_2} (\tilde{F}_{D_4 '})), 
\end{equation*}
and the result follows by Claim \ref{Claim projections}.

\textbf{Remaining proofs} The rest of the proofs are similar to those appearing in (\cite{wu2016proof}, Proposition 7.1). In particular, Lemma \ref{Lemma 7.2} is needed to prove part (5), and the remaining case of part (2) follows by an argument similar to Claim \ref{Claim projections}. In each case we get a constant $C$ multiplying $\epsilon$ and $\frac{1}{k}$ that does not depend on $k$ or $\epsilon$. Taking the maximal such constant, we obtain Proposition \ref{Proposition 7.1}. We omit the rest of the details.

\subsection{Proof of Theorem \ref{Theorem main Theorem}}
We begin by  relating our assumptions from Theorem \ref{Theorem main Theorem} to those of Proposition \ref{Lemma CP chain}, and hence to the subsequent results.
\begin{Lemma} \label{Lemma assumptions}
Let $\ell \subset \mathbb{R}^2$ be a non-principal line of positive slope. Suppose that for some $(\omega,\eta)\in [n_1]^\mathbb{N} \times [n_2]^\mathbb{N}$ we have
\begin{equation*}
\gamma := \overline{\dim}_B \left (\Fomega \times \Eeta\right) \bigcap \ell   >0.
\end{equation*}
Then $\exists (t_0, \omega_0,\eta_0)\in \mathbb{T} \times [n_1]^\mathbb{N} \times [n_2] ^\mathbb{N}$ and a line $\ell'$ of slope $m_1 ^{t_0}$ such that
\begin{equation} \label{Eq assumption}
\overline{\dim}_B \left( \pi_{m_1} ( \tilde{F}_{\omega_0}) \times \pi_{m_2} ( \tilde{E}_{\eta_0}) \right) \cap \ell' \geq \gamma >0.
\end{equation}
Moreover, if $\alpha_1 \times \alpha_2 \in P( [n_1]^\mathbb{N}\times [n_2]^\mathbb{N})$ is a product of Bernoulli measures, then there is a set $A$ of full $\alpha_1  \times \alpha_2$  measure such that:  If $(\omega,\eta)\in A$ then we may take $(t_0,\omega_0,\eta_0)$ to be generic with respect to the measure preserving system $(\mathbb{T} \times [n_1]^\mathbb{N}\times [n_2]^\mathbb{N},Z, \lambda\times \alpha_1 \times \alpha_2)$, discussed in Theorem \ref{Theorem X ergodic MPS}.
\end{Lemma}

The case of a  negative slope can be treated in a completely analogues way.  We defer the proof to section \ref{Section Remaining proofs}.

Now, we want to show that if \eqref{Eq assumption} holds then $\gamma+1 \leq \max_{(i,j)\in [n_1]\times [n_2]} \frac{\log |\Gamma_i|}{\log m_1} + \frac{\log |\Lambda_j|}{\log m_2}$. We also want to show that under the additional assumption that $(\omega_0,\eta_0)$ are typical with respect to $t_0$ and a product of Bernoulli measures $\alpha_1\times \alpha_2 \in P( [n_1]^\mathbb{N}\times [n_2]^\mathbb{N})$ (in the sense of Theorem \ref{Theorem X ergodic MPS}), then
\begin{equation*}
1+\gamma \leq \sum_{i=0} ^{n_1-1} \frac{\log |\Gamma_i|}{\log m_1}\cdot \alpha_1 ([i])+\sum_{i=0} ^{n_2-1} \frac{\log |\Lambda_i|}{\log m_2}\cdot \alpha_2 ([i]).
\end{equation*}
We shall prove the latter assertion. The other assertion follows from a similar argument. 

For this end, let $\epsilon>0$, and let $r_0 = r_0 (\epsilon)$, $k_5 = k_5 (\epsilon)$ be as in Proposition \ref{Proposition 7.1}. Let $k\geq k_5$. Choose a point $(z,t,\omega,\eta) \in [0,1]^2 \times \mathbb{T} \times [n_1]^\mathbb{N} \times [n_2] ^\mathbb{N}$, a measure $\nu \in P([0,1]^2)$,  a set $D \subset [0,1]^2 \times \mathbb{T} \times [n_1]^\mathbb{N} \times [n_2] ^\mathbb{N}$ and $\mathcal{N} \subset \mathbb{N}$ with the properties stated in Proposition 7.1.

\begin{Lemma} \label{Equation 7.5}
For all $p\in \mathcal{N}$,
\begin{equation} 
\inf_{y\in \mathbb{R}^2} \frac{1}{k\log 2} \log N_{2^{-k}} ( \supp (\nu^{\mathcal{A}_p ^t (z)}) \setminus B(y,r_0)) \geq \gamma -o(1), \quad \text{ as }  \epsilon\rightarrow 0 \text{ and } k \rightarrow \infty.
\end{equation}
\end{Lemma}
\begin{proof}
This is a consquence of property (5) of Proposition \ref{Proposition 7.1} as proven in \cite{wu2016proof}, equation (7.3).
\end{proof}

Let $K := \pi_{m_1} \times \pi_{m_2} (\tilde{D'}_{3,4})$ be a union of product sets, where $D_{3,4} ' = \bigcup_{p\in \mathcal{N}} \Pi_{3,4} U^p (z,t,\omega,\eta) $ and $\tilde{D'}_{3,4}\subseteq [m_1]^\mathbb{N} \times [m_2]^\mathbb{N}$  is as in Proposition \ref{Proposition 7.1} part (6) (and part (2)).

\begin{Lemma} \label{D and D}
For all $p\in \mathcal{N}$, $\nu^ {A_p ^t (z)}$ is supported on a slice $\ell' \cap K$ of $K$ of slope $m^{\Rtheta ^p (t)}$, and $U_t ^p (z) \in \supp ( \nu^ {A_p ^t (z)}) \cap D_1$ is a point on this line. In particular,
\begin{equation} \label{7.5}
\inf_{y\in \mathbb{R}^2} \frac{1}{k\log 2} \log N_{2^{-k}} (\ell' \cap K  \setminus B(y,r_0)) \geq \gamma-o(1), \quad \text{ as }  \epsilon\rightarrow 0 \text{ and } k \rightarrow \infty.
\end{equation}
\end{Lemma}

\begin{proof}
Since $\nu \in P( \left( \Fomega\times \Eeta \right) \bigcap \ell)$ for a line $\ell$ of slope $m_1 ^t$, the measure $\nu^ {A_p ^t (z)}$ is supported on a slice $ \left( \pi_{m_1}(\tilde{F}_{\sigma_t ^p (\omega)}) \times \pi_{m_2}(\tilde{E}_{\sigma ^p (\eta)}) \right)\bigcap \ell'$, where $\ell'$ has slope $m^{\Rtheta^p (t)}$.  In addition, for every $p\in \mathcal{N}$, $\Pi_1(U^p (z,t,\omega,\eta)) \in \Pi_1 (D)$ and $\Pi_{3,4} (U^p (z,t,\omega,\eta))=(\sigma_t ^p (\omega), \sigma ^p (\eta)) \in D_{3,4} '$. So,
\begin{equation*}
\supp (\nu^{\mathcal{A}_p ^t (z)}) \subseteq \ell' \cap \left( \pi_{m_1}(\tilde{F}_{\sigma_t ^p (\omega)}) \times \pi_{m_2}(\tilde{E}_{\sigma ^p (\eta)}) \right) \subseteq \ell'\cap K,
\end{equation*}
and $U_t ^p (z) \in \supp (\nu^{A_p ^t (z)}) \cap \Pi_1 (D) \neq \emptyset$, since $z \in \supp (\nu)$. The last assertion is thus a consequence Lemma \ref{Equation 7.5}.
\end{proof}

To sum up, for every $\epsilon>0$ and large enough $k$, we have produced sets $K \subset [0,1]^2$, $D_1 \subset [0,1]^2$ and $V = \lbrace \Rtheta ^p (t) : p\in \mathcal{N} \rbrace$ such that:
\begin{enumerate}
\item $\lambda(\overline{V}) \geq 1- o(1), \quad \text{ as }  \epsilon\rightarrow 0 \text{ and } k \rightarrow \infty.$

\item We have
\begin{equation*}
 \frac{1}{k} \log N_{2^{-k}} (D_1) = o(1), \quad \text{ as }  \epsilon\rightarrow 0 \text{ and } k \rightarrow \infty.
\end{equation*}

and,
\begin{equation*}
\frac{1}{k \log 2} \log N_{2^{-k}} (K)  \leq  \sum_{i=0} ^{n_1-1} \frac{\log |\Gamma_i|}{\log m_1}\cdot \rho ([i])+\sum_{i=0} ^{n_2-1} \frac{\log |\Lambda_i|}{\log m_2}\cdot \alpha ([i])+o(1), \quad \text{ as }  \epsilon\rightarrow 0 \text{ and } k \rightarrow \infty.
\end{equation*}

\item By Lemma \ref{D and D}, for every $v\in V$ there is a line $\ell'$ of slope $m^v$ that intersects $D_1 \cap  K$, such that \eqref{7.5} holds. 
\end{enumerate}

 Let $K' = K - D_1 = \lbrace k-d:k\in K, d\in D_1 \rbrace$. By items (1) and (3) above, for all $v\in V$ there is a line $\ell$ of slope $m^v$ satisfying \eqref{7.5}, passing through sufficiently many  $k$-level dyadic cubes containing the origin. Thus,
\begin{equation*}
\frac{\log N_{2^{-k}} ( K')}{ k\log 2} \geq 1+\gamma-o(1), \quad \text{ as }  \epsilon\rightarrow 0 \text{ and } k \rightarrow \infty.
\end{equation*}
Since for any two sets $A,B \subset \mathbb{R}^2$ there is a constant $C(2)$ such that \begin{equation*}
N_{2^{-k}}(A+B)\leq C(2)N_{2^{-k}} (A)\cdot N_{2^{-k}} (B).
\end{equation*}
We deduce from item (2) above that
\begin{equation*}
\sum_{i=0} ^{n_1-1} \frac{\log |\Gamma_i|}{\log m_1}\cdot \alpha_1 ([i])+\sum_{i=0} ^{n_2-1} \frac{\log |\Lambda_i|}{\log m_2}\cdot \alpha_2 ([i]) \geq \frac{\log N_{2^{-k}} ( K')}{ k\log 2} \geq 1+\gamma-o(1), \quad \text{ as }  \epsilon\rightarrow 0 \text{ and } k \rightarrow \infty.
\end{equation*}
Taking $\epsilon \rightarrow 0$ and $k\rightarrow \infty$, yields the Theorem.

\section{Remaining proofs} \label{Section Remaining proofs}

\textbf{Proof of Claim \ref{Claim tau t}}
\begin{proof}
For the first item, if $\tau=v_t$ for some $t\in \mathbb{T}$ then the existence of such a constant $C$ is well known. Otherwise,  $\tau = \lim_p v_{t_p}$ for some $t_p\in \mathbb{T}$, and let $k\in \mathbb{N}$. Find $p_0$ such that for all $p>p_0$, $d_\theta (t_p,\tau)<m_2 ^{-k}$. This means that the first $k$ digits of $\tau$ agree with the first $k$ digits of $v_{t_p}$ for all $p>p_0$. For any such $p$
\begin{equation*}
|r_k (\tau) - k \cdot \theta| \leq |r_k (\tau) - r_k (t_p)| + |r_k (t_p) - k\cdot \theta| \leq 0 +C
\end{equation*}
As required.

For the second item, let $p\in \mathbb{N}$ and let $\mathcal{C}_p (t)$ be the unique element of the partition $\mathcal{C}_p$ that contains $t$. Since $t$ is not an endpoint of $\mathcal{C}_p (t)$, $t$ belongs to the interior of that interval. Since $t_k$ converges to $t$, there is some $k_0$ such that for all $k>k_0$, $t_k$ also belongs to the interior of $\mathcal{C}_p (t)$. By noting that this means that $v_t$ and $v_{t_k}$ share the same first $p$ digits, we see that $d_\theta (v_{t_k},v_t)\leq \frac{1}{n^p}$, which is sufficient for the claim.
\end{proof}

\textbf{Proof of Lemma \ref{Lemma 5.1}}

\begin{proof}
Part (1) is an immediate corollary of \eqref{equation for A_kt}. For part (2), notice that $\pi_{v_t,\omega,\eta} \mu$ is a measure supported on some slice of $\Fomega\times \Eeta \subseteq [0,1]^2$, of the form $\ell_{m_1 ^t,z}$ for $z\in \Fomega\times \Eeta$. Notice that for every $k\geq 1$ for every atom $A\in \mathcal{A}_k ^t$, $A$ being a rectangle and $\supp(\pi_{v_t,\omega,\eta} \mu)$ being contained on a line, we have $| \supp(\pi_{v_t,\omega,\eta} \mu) \cap  \partial A| \leq 2$. As $\mu$ is not atomic, and the map $\pi_{v_t,\omega,\eta}$ is finite to one, $\pi_{v_t,\omega,\eta} \mu (\partial A) =0$. It follows that for $\mu$ a.e. $x$ and all $k\geq 1$ we have 
\begin{equation} \label{Eq we see}
\pi_{v_t,\omega,\eta} (\mu |_{[x_0 ^{k-1}]}) = \pi_{v_t,\omega,\eta} \mu |_{\mathcal{A}_k ^t (\pi_{v_t,\omega,\eta} (x))}.
\end{equation}
Finally, for all $t\in \mathbb{T}, \tau\in S, (\omega,\eta) \in ([n_1]\times [n_2]) ^\mathbb{N}$ and $x\in X_{t,\omega}$ we have (by the proof of Lemma \ref{Lemma 4.3})
\begin{eqnarray*}
U^k ( \pi_{\tau,\omega,\eta} (x), t, \omega,\eta)& =& ( U^k _t (\pi_{\tau,\omega,\eta} (x)), \Rtheta^k (t), \sigma_t ^k (\omega),\sigma^k (\eta) )\\
&=& (\pi_{\sigma^k (\tau), \sigma_t ^k (\omega), \sigma^k (\eta)} (\sigma^k (x)), \Rtheta^k (t), \sigma_t ^k (\omega),\sigma^k (\eta) )
\end{eqnarray*}
combining the last two calculation yields part (2) of the Lemma.
\end{proof}

\textbf{Proof of Proposition \ref{Proposition properties of nu}}

\begin{proof}
Part (1) is an easy consequence of the fact that, as $k$ grows to infinity, the maximal diameter of an element in the partition $\mathcal{B}_k$ converges to $0$. For the second part, let $k\in \mathbb{N}$ and fix and element  $A\times C \times I\times J \in \mathcal{B}_k$  where $J\in \mathcal{I}_2 ^k$,  $C\in \mathcal{C}_k$, $A\in \mathcal{A}_k ^t$ and $I\in \mathcal{I}_k ^t$, for some $t\in C$. Since $\partial (I) = \partial(J) =\emptyset$, we have $\partial (I\times J) = \emptyset$, since
\begin{equation*}
\partial(I\times J)= (\partial(I)\times J)\bigcup(I\times \partial(J)) = \emptyset.
\end{equation*}
Therefore, by two application of  the "product rule" for boundary of product sets
\begin{equation*}
\partial (A\times C \times I\times J ) \subseteq \partial (A\times C) \times [n_1]^\mathbb{N} \times [n_2]^\mathbb{N} = (\partial A \times C \times [n_1]^\mathbb{N} \times [n_2]^\mathbb{N}) \bigcup (A \times \partial C \times [n_1]^\mathbb{N} \times [n_2]^\mathbb{N}).
\end{equation*} 
Thus, by Boole's inequality,
\begin{equation*}
\nu_\infty (\partial (A\times C \times I\times J ))\leq \nu_\infty (\partial A\times C \times [n_1]^\mathbb{N} \times [n_2]^\mathbb{N}) + \nu_\infty (A \times \partial C  \times [n_1]^\mathbb{N} \times [n_2]^\mathbb{N}).
\end{equation*}
Now, the first summoned on the right hand side above is $0$. This is because $Q^{(\mu_0,x_0,t_0,\tau_0, \omega_0,\eta_0)}_1$ a.e. $\mu$ has positive and exact dimension. By Lemma \ref{Lemma 4.1} it follows that $\pi_{\tau,\omega,\eta} \mu$ is also exact dimensional with positive dimension, and is therefore not atomic. It is also supported on a line, and $\partial A$ is a union of four lines. Thus, $\pi_{\tau,\omega,\eta} \mu (\partial A)=0$ almost surely (this is not too different from the proof of Lemma \ref{Lemma 5.1}). The second summoned is trivially $0$ since the marginal on the second coordinate of $\nu_\infty$ is $\lambda$, and $\partial C$ consists of two points.

For part (3), we make use of Corollary \ref{Cor joint dist}. By this Corollary, and our assumptions, we have that the joint distribution of $Q$ on $\mathbb{T} \times [n_1]^\mathbb{N} \times [n_2] ^\mathbb{N}$ is $\lambda\times \alpha_1 \times \alpha_2$. So,  we may assume we chose $Q^{(\mu_0,x_0,t_0,\tau_0 , \omega_0,\eta_0)}$ so that it gives full mass to the set
\begin{equation*}
\lbrace (\mu,x,t,\tau, \omega,\eta): \tau=v_t, \quad (\omega,\eta) \text{ are } \alpha_1\times \alpha_2 \text{ generic } \rbrace. 
\end{equation*}
This can be done since $Q$ gives this set full measure.

For the last part, by the definition of Kolmogorov-Sinai entropy, it suffices to show that there exists some partition $\mathcal{P}_1$ of $[0,1]^2 \times \mathbb{T} \times [n_1]^\mathbb{N} \times [n_2] ^\mathbb{N}$ s.t. $\lim_k \frac{1}{k} H(\nu_\infty, \bigvee_{i=0} ^{k-1} U^{-i} \mathcal{P}_1) >0$. Consider the partition $P_1 = [\mathcal{D}_{m_1} \times \mathcal{D}_{m_2}] \times \mathcal{C} \times [n_1]^\mathbb{N} \times [n_2] ^\mathbb{N}$ (where the elements in this partition were defined in subsection \ref{Section 5.1}).  Denote, for every $k$, $\mathcal{P}_k = \bigvee_{i=0} ^{n-1} U^{-i} \mathcal{P}_1$.

Denoting by $\Pi_{1,2} :[0,1]^2 \times \mathbb{T} \times [n_1]^\mathbb{N} \times [n_2] ^\mathbb{N} \rightarrow [0,1]^2 \times \mathbb{T} $ the coordinate projection, and $\Pi_1$ similarly, we have
\begin{eqnarray*}
H( \nu_\infty, \mathcal{P}_k) &=& \sum_{B \in \Pi_{1,2} (\mathcal{B}_k)}  \nu_\infty (B\times [n_1]^\mathbb{N} \times [n_2] ^\mathbb{N}) \cdot \log (\nu_\infty (B\times [n_1]^\mathbb{N}\times [n_2] ^\mathbb{N})^{-1}\\
&=& \sum_{B \in \Pi_{1,2} (\mathcal{B}_k)} \Pi_{1,2} \nu_\infty (B) \cdot \log (\Pi_{1,2} \nu_\infty (B))^{-1} \\
\end{eqnarray*}
Denote the measure $\Pi_{1,2} \nu_\infty$ by $\nu\in P([0,1]^2 \times \mathbb{T})$.  Define  partitions of $[0,1]^2 \times \mathbb{T}$ by $\mathcal{W}_k  = \bigvee_{i=0} ^{n-1} {U}^{-i} _{1,2} ([\mathcal{D}_m \times \mathcal{D}_n] \times \mathcal{C})$, where ${U}_{1,2}$ is the restriction of $U$ to $[0,1]^2 \times \mathbb{T}$.  Thus,
\begin{equation*}
H( \nu_\infty, \mathcal{P}_k) = H(\nu , \mathcal{W}_k).
\end{equation*}
The proof that $\lim_k \frac{1}{k} H( \nu, \mathcal{W}_k) >0$, which proves our claim, is now quite similar to the proof of Proposition 5.8 in \cite{wu2016proof}. We  omit the details. 
\end{proof}

\textbf{Proof of Lemma \ref{Lemma assumptions}}
\begin{proof}
Define functions $\phi_1,\phi_2 :[0,1]^2 \rightarrow [0,1]^2$ by
\begin{equation*}
\phi_1 (x,y) = (T_{m_1} (x), T_{m_2}(y)),\quad \phi_2 (x,y)= (x,T_{m_2} (y)).
\end{equation*}
Let $u>0$ denote the slope of $\ell$. Then $\phi_1 ( \ell )$ is a finite family of lines through $[0,1]^2$, all with slope $u\cdot \frac{m_2}{m_1}$, and at least one of these lines  intersects $\pi_{m_1} ( \tilde{F}_{\sigma(\omega)}) \times \pi_{m_2} (\tilde{E}_{\sigma(\eta)})$ in a set of upper box dimension $\geq \gamma$. Similarly, $\phi_2 (\ell)$ is a finite family of lines through $[0,1]^2$, all with slope $u\cdot m_2$,  and at least one of these lines  intersects $\pi_{m_1} (\tilde{F}_{\omega}) \times \pi_{m_2} (\tilde{E}_{\sigma(\eta)})$ in a set of upper box dimension $\geq \gamma$.

Since $\frac{\log m_2}{\log m_1} \notin \mathbb{Q}$ the set $\lbrace u\cdot \frac{m_2 ^k}{m_1 ^n}: k\geq n \rbrace$ is dense in $(0,\infty)$. Therefore, there exists $k\geq n$ such that $u\cdot \frac{m_2 ^k}{m_1 ^n} = m_1 ^{t_0} \in (1,m_1)$. By $n$ applications of $\phi_1$ to $\ell$ followed by $k-n$ applications $\phi_2$ to the resulting line, we see that there exists a line $\ell'$ of slope $m_1 ^{t_0}$ that intersects $\pi_{m_1} ( \tilde{F}_{\sigma^n (\omega)}) \times \pi_{m_2} (\tilde{E}_{\sigma^k (\eta)})$ in a set of dimension $\geq \gamma$. Denote $\omega_0 = \sigma^n (\omega), \eta_0 = \sigma^k (\eta)$. 

Finally, by Theorem \ref{Theorem X ergodic MPS}, there is a set $A' \subseteq [n_1]^\mathbb{N} \times [n_2]^\mathbb{N}$ satisfying $\alpha_1 \times \alpha_2 (A')=1$ such that for every $(\xi,\zeta)\in A'$, $(t_0,\xi,\zeta)$ is generic with respect to the system $(\mathbb{T} \times [n_1]^\mathbb{N}\times [n_2]^\mathbb{N},Z, \lambda\times \alpha_1 \times \alpha_2)$. Define $A = \sigma^{-n} \times \sigma^{-k} (A')$. Then since $\sigma^n \alpha_1=\alpha_1$ and $\sigma^k \alpha_2=\alpha_2$, the product $\sigma^k \times \sigma^n$ preserves the measure $\alpha_1 \times \alpha_2$. Therefore, $\alpha_1 \times \alpha_2 (\sigma^{-n} \times \sigma^{-k} (A'))=1$.   Finally, if $(\omega,\eta)\in A$ then $(\omega_0,\eta_0) = (\sigma^n (\omega),\sigma^k (\eta))\in A'$, so $(\omega_0,\eta_0)$ satisfies that $(t_0,\omega_0,\eta_0)$ is generic. 
\end{proof}

\bibliography{bib}{}
\bibliographystyle{plain}

\end{document}